\documentclass{amsart}
\usepackage{graphicx}
\usepackage{amssymb}
\usepackage{wrapfig}
\usepackage[bookmarksnumbered,colorlinks,plainpages,backref]{hyperref}

\newtheorem{thm}{Theorem}[section]
\newtheorem{cor}[thm]{Corollary}
\newtheorem{lem}[thm]{Lemma}
\newtheorem{prop}[thm]{Proposition}
\theoremstyle{definition}
\newtheorem{defn}[thm]{Definition}
\theoremstyle{remark}
\newtheorem{rem}[thm]{Remark}

\numberwithin{equation}{section}
\newcommand{\norm}[1]{\left\Vert#1\right\Vert}
\newcommand{\abs}[1]{\left\vert#1\right\vert}
\newcommand{\set}[1]{\left\{#1\right\}}

\newcommand{\dbar}{\bar\partial}
\newcommand{\ddbar}{\partial\bar\partial}

\DeclareMathOperator{\re}{Re}
\DeclareMathOperator{\im}{Im}

\DeclareMathOperator{\dist}{dist}

\begin{document}

\title[Sobolev Regularity of the Bergman Projection]{Sobolev Regularity of the Bergman Projection on a Smoothly Bounded Stein Domain that is not Hyperconvex}%
\author{Phillip S. Harrington}%
\address{SCEN 309, 1 University of Arkansas, Fayetteville, AR 72701}%
\email{psharrin@uark.edu}%


\begin{abstract}
  For every $0<r<\frac{1}{2}$, we will construct a flat K\"ahler manifold $M$ and a relatively compact domain with smooth boundary $\Omega\subset M$ that is Stein but not hyperconvex such that the Bergman projection $P$ on $\Omega$ is regular in the $L^2$ Sobolev space $W^s(\Omega)$ for all $0\leq s<r$ but irregular in $W^r(\Omega)$.  On these domains, we will also construct $f\in C^\infty(\overline\Omega)$ such that $Pf\notin C^\infty(\overline\Omega)$.  We will prove the same result for the invariant Bergman projection on $(2,0)$-forms.  These domains are modelled on a construction of Diederich and Ohsawa.
\end{abstract}
\maketitle



\section{Introduction}

Given a Hermitian manifold $M$ of complex dimension $n\geq 2$, a domain $\Omega\subset M$, and an integer $0\leq p\leq n$, the Bergman space of $(p,0)$-forms on $\Omega$ is the space of $L^2$ holomorphic $(p,0)$-forms on $\Omega$, denoted $A^2_{p,0}(\Omega)$.  The Bergman projection $P_\Omega^p$ is the orthogonal projection from $L^2_{p,0}(\Omega)$ to the Bergman space $A^2_{p,0}(\Omega)$.  We will focus on the case in which $\Omega$ is a relatively compact Stein domain, so that the Bergman space is infinite dimensional and independent of the choice of smooth Hermitian metric on $M$.  When $p=n$, the Bergman projection itself is independent of the choice of Hermitian metric, so we will refer to $P_\Omega^n$ as the invariant Bergman projection.  Furthermore, if the boundary of $\Omega$ is at least Lipschitz, there is a robust theory for the $L^2$ Sobolev spaces $W^s(\Omega)$ for all $s\in\mathbb{R}$ (see the survey in the first chapter of \cite{Gri85}, for example).  We will consider two fundamental questions:
\begin{enumerate}
  \item (Sobolev Regularity) Given $s>0$, is the Bergman projection $P_\Omega^p$ continuous in the $L^2$ Sobolev space $W^s(\Omega)$?
  \item (Global Regularity) Given $f\in C^\infty(\overline\Omega)$, is it always the case that $P_\Omega^p f\in C^\infty(\overline\Omega)$?
\end{enumerate}
A third important question is that of $L^q$ regularity for $q\neq 2$, but this is outside the scope of the present paper.

Both questions have been extensively studied when $M=\mathbb{C}^n$ for $n\geq 2$, beginning with the pioneering work of Kohn \cite{Koh78,Koh79}.  In this setting, regularity is independent of $p$, so it is common to focus on the $p=0$ case.  See \cite{BoSt99}, \cite{ChSh01}, or \cite{Str10} for detailed discussion of the key results.  Most of the results considered in these sources consider sufficient conditions for global regularity of the Bergman projection on smooth, relatively compact, Stein domains in $\mathbb{C}^n$.  We note that regularity for the Bergman projection on domains in $\mathbb{C}^n$ is often considered in parallel with regularity for the $\bar\partial$-Neumann operator $N_\Omega^{p,q}$, a connection that was made explicit by Kohn's formula $P_\Omega^0=I-\dbar^* N_\Omega^{0,1}\dbar$ and further developed by Boas and Straube in \cite{BoSt90}.  In the negative direction, Barrett proved in \cite{Bar92} that for any $s>0$, there exists a relatively compact, smooth, Stein domain in $\mathbb{C}^2$ on which Sobolev regularity fails for the Bergman projection in $W^s(\Omega)$.  These domains are the worm domains of Diederich and Forn{\ae}ss \cite{DiFo77a}.  Building on Barrett's work, Christ proved that global regularity fails for the Bergman projection on the worm domain \cite{Chr96}.

Given the failures of regularity on the worm domain, we next consider for which values of $s$ the Bergman projection exhibits Sobolev regularity.  In \cite{BeCh00}, Berndtsson and Charpentier provided a positive lower bound for the optimal Sobolev regularity of the Bergman projection on any relatively compact domain in $\mathbb{C}^n$ with Lipschitz boundary.  As in Kohn's work in \cite{Koh99} (itself inspired by work of Boas and Straube in \cite{BoSt91}), Berndtsson and Charpentier use the Diederich-Forn{\ae}ss Index: the supremum over all exponents $0<\eta<1$ such that there exists a defining function $\rho$ for $\Omega$ such that $-(-\rho)^\eta$ is strictly plurisubharmonic on $\Omega$.  Diederich and Forn{\ae}ss proved that this index is positive on relatively compact domains $\Omega$ with $C^2$ boundaries in Stein manifolds $M$ \cite{DiFo77b}, and the author extended Diederich and Forn{\ae}ss's result to domains with Lipschitz boundaries \cite{Har08a}.  If $DF(\Omega)$ denotes the Diederich-Forn{\ae}ss index of a relatively compact domain in $\mathbb{C}^n$ with Lipschitz boundary, Berndtsson and Charpentier proved that the Bergman projection is regular in $W^s(\Omega)$ for all $0<s<\frac{1}{2}DF(\Omega)$.  In \cite{Liu19a}, B.~Liu computed the precise value of the Diederich-Forn{\ae}ss index on the worm domain, so the result of Berndtsson and Charpentier can be used to provide a lower bound for Sobolev regularity of the Bergman projection on the worm domain.


Domains with a positive Diederich-Forn{\ae}ss index are examples of hyperconvex domains, i.e., domains admitting a bounded, strictly plu\-ri\-sub\-har\-mo\-nic exhaustion function.  Given the results establishing So\-bo\-lev regularity for the Bergman projection on hyperconvex domains, the goal of the present paper is to study regularity of the Bergman projection on non-hyperconvex Stein domains.  In $\mathbb{C}^n$, such domains must be either unbounded or have non-Lipschitz boundaries by \cite{DiFo77a} or \cite{Dem87}, as discussed above.  Indeed, the classic example of a non-hyperconvex Stein domain is the Hartogs triangle $\mathbb{H}=\{(z_1,z_2):|z_1|<|z_2|<1\}$, which is not Lipschitz on any neighborhood of $(0,0)$.  Since the boundary is not Lipschitz, care is required even to define $W^s(\Omega)$ when $s$ is not a natural number, so most results focus on regularity in Sobolev spaces of integer order.  There is a rich body of work on the Hartogs triangle (see the survey in \cite{Sha15}), of which we highlight a few papers that are particularly relevant to our questions of Sobolev regularity.  Sobolev regularity for the Bergman projection on the Hartogs triangle was studied by Chakrabarti and Shaw in \cite{ChSh13}, but this required the inclusion of a weight function to tame the singularity at the origin.  L.~Chen \cite{Che17b} studied Sobolev regularity in $L^q$ Sobolev spaces on the Hartogs triangle, but these also required the inclusion of a weight function.  Edholm and McNeal studied $L^q$ Sobolev regularity of the Bergman projection on a generalization of the Hartogs triangle in \cite{EdMc20}.  They showed that Sobolev regularity fails in unweighted Sobolev spaces on these generalized Hartogs triangles (for derivatives of integer order), although they did prove regularity when the first derivatives are in $L^q$ for $\frac{4}{3}<q<2$.  Higher dimensional generalizations of the Hartogs triangle have been considered by Zhang in \cite{Zha21a} and by H.~Y.~Liu and Tu in \cite{LiTu24}.  

In \cite{Har23}, the author noted that the Diederich-For\-n{\ae}ss index is trivial on a family of relatively compact Stein domains with smooth boundaries constructed by Diederich and Ohsawa \cite{DiOh82}, with more explicit representations given by Barrett \cite{Bar86} and Kiselman \cite{Kis91}.  The worm domain of Diederich and Forn{\ae}ss contains an analytic annulus in the boundary, which is the obstruction to regularity of the Bergman projection.  The role of the annulus is made explicit in the analysis carried out by Boas and Straube in \cite{BoSt93}.  The domains constructed in \cite{DiOh82}, \cite{Bar86}, and \cite{Kis91} connect the inner and outer boundaries of this annulus to create a compact Riemann surface of genus one in the boundary, which is known to be an obstruction to positivity of the Diederich-Forn{\ae}ss index (Proposition 5.7 in \cite{AdYu21}).  We will see in Corollary \ref{cor:hyperconvex} below that these domains are not even hyperconvex, i.e., there exists no bounded plurisubharmonic exhaustion function for these domains, and hence the weaker notions of the Diederich-Forn{\ae}ss index considered by Adachi and Yum in \cite{AdYu21} will also vanish.  In \cite{Che17a}, B.-Y.~Chen introduced a generalization of the Diederich-Forn{\ae}ss index which is useful in the study of the Bergman kernel, but even this hyperconvexity index must vanish on the domains under consideration. 

In the present paper, we will construct a family of relatively compact Stein domains with smooth boundaries that are not hyperconvex.  Our construction is modelled on the construction given by Diederich and Ohsawa \cite{DiOh82}, but it has more in common with the relatively explicit construction given by Barrett \cite{Bar86} and Kiselman \cite{Kis91}.  Like Barrett and Kiselman, our family of domains will depend on a parameter, and we will see that this parameter governs the Sobolev regularity of the Bergman projection.  In contrast to the Hartogs triangle, we will see that the Bergman projection on these domains exhibits regularity properties analogous to those known for the worm domain.  Our results are sharp: we can state the precise value of the optimal Sobolev regularity for each domain in our family.  This gives us the following result:
\begin{thm}
\label{thm:main}
  For every $0<r<\frac{1}{2}$ and $p\in\{0,1,2\}$, there exists a Hermitian manifold $M$ of complex dimension $2$ and a domain $\Omega\subset M$ such that
  \begin{enumerate}
    \item $M$ is a flat K\"ahler manifold,
    \item $\Omega$ is relatively compact in $M$,
    \item $\Omega$ is Stein,
    \item the boundary of $\Omega$ is real-analytic,
    \item $\Omega$ is not hyperconvex,
    \item the Bergman projection $P_\Omega^p$ is continuous in $W^s_{p,0}(\Omega)$ for all $0\leq s<r$,
    \item there exists $f\in C^\infty_{p,0}(\overline{\Omega})$ such that $P_\Omega^p f\notin W^r_{p,0}(\Omega)$.
  \end{enumerate}
\end{thm}
Note that the final part of Theorem \ref{thm:main} proves that our Sobolev regularity results are sharp on $\Omega$ and that global regularity fails on $\Omega$ (refining a result of Kiselman \cite{Kis91}).  Theorem \ref{thm:main} will follow from Propositions \ref{prop:M_properties} and \ref{prop:Omega_properties} and Corollaries \ref{cor:hyperconvex} and \ref{cor:sharp_regularity}, below.  As noted above, the choice of smooth metric on $M$ does not impact the Bergman space $A^2_{p,0}(\Omega)$ when $\Omega$ is relatively compact.  However, it does impact the Bergman projection itself when $p\neq n$.  Given that $\Omega$ is not hyperconvex but is relatively compact with smooth boundary, $M$ cannot be Stein, so a flat K\"ahler manifold, i.e., a manifold that is locally isometric to Euclidean space, is the closest we can come to the customary setting of $\mathbb{C}^n$. 

We will see in Theorems \ref{thm:regularity} and \ref{thm:irregularity} below that $P_\Omega^1$ and $P_\Omega^2$ will have the same regularity properties on the domains under consideration.  However, unless $r=\frac{1}{k}$ for some integer $k>2$, $P_\Omega^0$ will exhibit strictly stronger Sobolev regularity properties.  It would be interesting to know whether regularity for the classical Bergman projection $P_\Omega^0$ with respect to an arbitrary Hermitian metric is always at least as strong as regularity for the invariant Bergman projection $P_\Omega^n$.  Furthermore, we will see in Corollary \ref{cor:sharp_regularity} that the parameter $\mu$ used in the construction of $M$ is related to the parameter $r$ given in Theorem \ref{thm:main} by the relations $r=\frac{1-\lfloor\mu\rfloor}{\mu}+1$ when $p=0$ and $r=\frac{1}{\mu}$ when $p\in\{1,2\}$.  It is interesting to note that a domain and metric that depend smoothly on a parameter $\mu$ exhibit regularity properties for $P_\Omega^0$ that depend discontinuously on this same parameter, and that this same discontinuity disappears when we consider the invariant Bergman projection $P_\Omega^2$.  Compare the phenomenon observed by Edholm and McNeal when considering $L^q$ regularity of the Bergman projection on a generalization of the Hartogs triangle in \cite{EdMc17}, in which the regularity properties depend on whether a parameter is rational or irrational.

Our method of proof involves constructing a biholomorphism between $\Omega$ and a Reinhardt domain $D$ given by $$D=\set{w\in\mathbb{C}^2:0<|w_1|<1\text{ and }\abs{\log|w_2|^2}<\arccos\left(|w_1|^{\mu}\right)},$$ where $\mu>0$ is a parameter used in the construction of $M$.  Although $\Omega$ itself has a smooth boundary, $D$ does not even have a Lipschitz boundary, which complicates the analysis considerably.  Indeed, this is an example of a biholomorphism that does not extend smoothly to the boundary between two domains that do not satisfy Condition R (see \cite{Bar86}).  We will push forward the metric on $\Omega$ to obtain a flat metric on $D$, and the natural automorphisms of the Reinhardt domain will be isometries in this new metric.  The volume element in this new metric will take the form $dV_{D}=\frac{4\mu^2|w_1|^{2\mu-2}}{|w_2|^2}dV_{\mathbb{C}^2}.$  Clearly, this metric is singular near boundary points at which $w_1=0$ (unless $\mu=1$), further complicating our analysis.  Sobolev regularity for the Bergman projection on Reinhardt domains has been considered by many authors, including the work on the Hartogs triangle discussed above.  As noted in \cite{ChSh13}, the Hartogs triangle is biholomorphic to a product domain, but our domain $D$ lacks this useful feature.  Global regularity for the Bergman projection on smoothly bounded Reinhardt domains was considered by Bell and Boas \cite{BeBo81} and Barrett \cite{Bar81}.  Sobolev regularity for the Bergman projection on smoothly bounded Reinhardt domains was considered by Boas \cite{Boa84} and Straube \cite{Str86}.  \v{C}u\v{c}kovi\'{c} and Zeytuncu consider regularity of the Bergman projection in weighted Sobolev spaces on Reinhardt domains \cite{CuZe16}, which is equivalent to working with the non-Euclidean metric that we consider, and Zeytuncu considered similar results on convex Reinhardt domains in \cite{Zey17}, but again all of these results require smooth boundaries, so they do not apply directly to our case.  

We note that many authors have considered regularity of the Berg\-man projection on variations of the worm domain.  We highlight recent work in \cite{KPS16} and \cite{KMPS23}, which consider generalizations of the worm domain which are unbounded in $\mathbb{C}^n$ (and non-smooth in \cite{KPS16}).  On such domains, the authors show that Sobolev regularity fails for all $s>0$.  This is in contrast to our result, which sacrifices hyperconvexity to preserve relative compactness of the domain, and still finds some Sobolev regularity of the Bergman projection.

The author would like to thank Debraj Chakrabarti for helpful conversations regarding the significance of the invariant Bergman projection $P_\Omega^n$.

\section{Basic Properties}

We first define the ambient manifold for our Stein domain.
\begin{defn}
\label{defn:M}
  Let $\tilde M=\mathbb{C}\times(\mathbb{C}\backslash\{0\})$ equipped with the K\"ahler form
  \begin{equation}
  \label{eq:Kahler_form}
    \omega_{\tilde M}=\frac{i}{2}dz_1\wedge d\bar z_1+\frac{i}{2|z_2|^2}dz_2\wedge d\bar z_2.
  \end{equation}
  For any $\mu>0$, let $M_\mu$ denote the quotient of $\tilde M$ by the equivalence relation $(z_1,z_2)\sim(e^{2\pi\mu i}z_1,e^{\pi\mu} z_2)$ endowed with the K\"ahler form $\omega_{M_\mu}$ induced by \eqref{eq:Kahler_form}.
\end{defn}

The basic properties of $M_\mu$ are as follows:
\begin{prop}
\label{prop:M_properties}
  For $M_\mu$ defined by Definition \ref{defn:M}, $M_\mu$ is a flat K\"ahler manifold that is neither compact nor Stein.
\end{prop}

\begin{proof}
  With the given K\"ahler form, $\tilde M$ is locally isometric to Euclidean space under the map $(z_1,z_2)\mapsto(z_1,\log z_2)$ for an appropriate branch of $\log z_2$, so $\tilde M$ is a flat manifold.  Since $M_\mu$ is locally isometric to $\tilde M$, $M_\mu$ is also flat.


  Set
  \begin{equation}
  \label{eq:tilde_S_defn}
    \tilde S=\{z\in\tilde M:z_1=0\}
  \end{equation}
  and
  \begin{equation}
  \label{eq:S_defn}
    S_\mu=\{[z]\in M_\mu:z\in\tilde S\}.
  \end{equation}
  Then $S_\mu$ is a compact Riemann surface of genus one.  Since $M_\mu$ contains a compact Riemann surface, it cannot be Stein.

\end{proof}

Now we may define our key domain:
\begin{defn}
\label{defn:Omega}
  Set
  \[
    \tilde\Omega=\set{z\in\tilde M:\abs{z_1+e^{i\log|z_2|^2}}^2<1}.
  \]
  For $\mu>1$, we may define
  \[
    \Omega_\mu=\set{[z]\in M_\mu:z\in\tilde\Omega}.
  \]
\end{defn}

As in \cite{Bar86} and \cite{Kis91}, $\Omega_\mu$ is an explicit formulation of an example constructed in \cite{DiOh82}.  Observe that $\mu$ is an integer in both Barrett's construction and Kiselman's construction, and hence they omit the rotation in the $z_1$ variable.  We will see that non-integer values of $\mu$ are essential in recovering the full range of Sobolev regularity for the Bergman projection. 
\begin{prop}
\label{prop:Omega_properties}
  For $\Omega_\mu\subset M_\mu$ defined by Definition \ref{defn:Omega}, $\Omega_\mu$ is a relatively compact Stein domain with smooth, real-analytic boundary, but $\overline\Omega_\mu$ admits no Stein neighborhood in $M_\mu$.
\end{prop}

\begin{proof}
  Let
  \begin{equation}
  \label{eq:tilde_rho_defn}
    \tilde\rho(z)=\abs{z_1+e^{i\log|z_2|^2}}^2-1
  \end{equation}
  be a defining function for $\tilde\Omega$.  Observe that $\tilde\rho(e^{2\pi\mu i}z_1,e^{\pi\mu} z_2)=\tilde\rho(z_1,z_2)$ for all $z\in\tilde M$, so $\tilde\rho$ induces a real-analytic function $\rho_\mu$ defining $\Omega_\mu$ in $M_\mu$.  To confirm that $\tilde\rho$ and $\rho_\mu$ are true defining functions, we must confirm that the gradients do not vanish on the boundaries of each domain.  To that end, we compute
  \[
    \partial\tilde\rho(z)=\left(\bar z_1+e^{-i\log|z_2|^2}\right)dz_1+\frac{2\im\left(z_1 e^{-i\log|z_2|^2}\right)}{z_2}dz_2.
  \]
  If $\frac{\partial\tilde\rho}{\partial z_1}(z)=0$, then $\bar z_1+e^{-i\log|z_2|^2}=0$, so $z_1=-e^{i\log|z_2|^2}$.  Since $\tilde\rho(-e^{i\log|z_2|^2},z_2)=-1$, we conclude that
  \begin{equation}
  \label{eq:nonvanishing_first_derivative}
    \frac{\partial\tilde\rho}{\partial z_1}(z)\neq 0\text{ whenever }\tilde\rho(z)=0,
  \end{equation}
  and hence $\partial\tilde\rho\neq 0$ on $\partial\tilde\Omega$.  This also implies that $\partial\rho_\mu\neq 0$ on $\partial\Omega_\mu$, so $\rho_\mu$ is a defining function for $\Omega_\mu$ and hence $\Omega_\mu$ has smooth boundary.
  

  Next, we compute
  \begin{multline*}
    \ddbar\tilde\rho(z)=dz_1\wedge d\bar z_1+\frac{i}{z_2}e^{i\log|z_2|^2}dz_2\wedge d\bar z_1-\frac{i}{\bar z_2} e^{-i\log|z_2|^2}dz_1\wedge d\bar z_2\\
    -\frac{2\re\left(z_1e^{-i\log|z_2|^2}\right)}{|z_2|^2}dz_2\wedge d\bar z_2.
  \end{multline*}
  If we set $x=\re\left(z_1 e^{-i\log|z_2|^2}\right)$ and $y=\im\left(z_1 e^{-i\log|z_2|^2}\right)$, then we may write
  \begin{gather*}
    \tilde\rho(z)=x^2+y^2+2x,\\
    \partial\tilde\rho(z)=\left(x-iy+1\right)e^{-i\log|z_2|^2}dz_1+\frac{2y}{z_2}dz_2,
  \end{gather*}
  and
  \begin{multline*}
    \ddbar\tilde\rho(z)=dz_1\wedge d\bar z_1+\frac{i}{z_2}e^{i\log|z_2|^2}dz_2\wedge d\bar z_1-\frac{i}{\bar z_2} e^{-i\log|z_2|^2}dz_1\wedge d\bar z_2\\
    -\frac{2x}{|z_2|^2}dz_2\wedge d\bar z_2.
  \end{multline*}
  On $\partial\tilde\Omega$, we set $Z=\frac{2y}{z_2}\frac{\partial}{\partial z_1}-\left(x-iy+1\right)e^{-i\log|z_2|^2}\frac{\partial}{\partial z_2}$, so that $Z$ spans $T^{1,0}(\partial\tilde\Omega)$.  We compute the Levi-form via
  \[
    \ddbar\tilde\rho(Z,\bar Z)(z)=-\frac{2x((x+1)^2+y^2)}{|z_2|^2}=\frac{2(-x)(\tilde\rho(z)+1)}{|z_2|^2}.
  \]
  When $\tilde\rho(z)=0$, $x\leq 0$, and hence the Levi-form of $\tilde\rho$ is non-negative on $\partial\tilde\Omega$.  This implies that $\tilde\Omega$, and hence $\Omega_\mu$, has a pseudoconvex boundary.  When $x<0$, $\partial\tilde\Omega$ is strictly pseudoconvex, and hence $\partial\Omega_\mu$ must have at least one strictly pseudoconvex point.  As shown in \cite{DiOh82}, this guarantees that $\Omega_\mu$ is Stein.

  For $\tilde S$ defined by \eqref{eq:tilde_S_defn}, $\tilde\rho|_{\tilde S}\equiv 0$, and hence $\tilde S\subset\partial\tilde\Omega$.  As a result, $S_\mu$ defined by \eqref{eq:S_defn} satisfies $S_\mu\subset\partial\Omega_\mu$.  Since $\partial\Omega_\mu$ contains a compact Riemann surface, $\overline{\Omega_\mu}$ admits no Stein neighborhood.
\end{proof}

We next observe that $\Omega_\mu$ in the given metric exhibits a high degree of symmetry, as the following lemma makes explicit:
\begin{lem}
\label{lem:isometry}
  For $\theta_1,\theta_2\in\mathbb{R}$, define a biholomorphic map $\tilde\psi:\tilde M\rightarrow\tilde M$ by
  \begin{equation}
  \label{eq:psi_defn}
    \tilde\psi_{\theta_1,\theta_2}(z)=\left(e^{i\mu\theta_1}z_1,e^{\frac{\mu}{2}\theta_1+i\theta_2}z_2\right).
  \end{equation}
  Then $\tilde\psi_{\theta_1,\theta_2}$ is an isometry on $\tilde M$ that induces a biholomorphic isometry $\psi_{\theta_1,\theta_2,\mu}:\Omega_\mu\rightarrow\Omega_\mu$.
\end{lem}
The proof is straightforward, so we omit the details.



We now prove the crucial result that $\Omega_\mu$ is biholomorphic to a Reinhardt domain, which we will denote $D_\mu$.
\begin{lem}
\label{lem:biholomorphism}
  For $t>0$ and $z\in\mathbb{C}\backslash\{0\}$, let $\log^t(z)$ denote the unique logarithm of $z$ satisfying
  \begin{equation}
  \label{eq:log_branch}
    0\leq\im\log^t(z)-\log t^2<2\pi.
  \end{equation}
  Then $\log^{|z_2|}(z_1)$ is holomorphic on $\tilde\Omega$ and locally constant in $z_2$.  If we set
  \begin{equation}
  \label{eq:D_defn}
    D_\mu=\set{w\in\mathbb{C}^2:0<|w_1|<1\text{ and }\abs{\log|w_2|^2}<\arccos\left(|w_1|^{\mu}\right)}
  \end{equation}
  and define $\tilde\varphi:\tilde\Omega\rightarrow D_\mu$ by
  \begin{equation}
  \label{eq:varphi_defn}
    \tilde\varphi(z)=\left(\exp\left(\frac{\log^{|z_2|}(z_1)-\log 2}{\mu}\right),z_2\exp\left(\frac{1}{2}\left(\pi+i \log^{|z_2|}(z_1)\right)\right)\right),
  \end{equation}
  then $\tilde\varphi$ induces a biholomorphism $\varphi_\mu$ between $\Omega_\mu$ and $D_\mu$.  For $\theta_1,\theta_2\in\mathbb{R}$, the isometry $\psi_{\theta_1,\theta_2,\mu}$ induced by \eqref{eq:psi_defn} satisfies
  \begin{equation}
  \label{eq:isometry_push_forward}
    \varphi_\mu(\psi_{\theta_1,\theta_2,\mu}([z]))=(e^{i\theta_1}(\varphi_\mu)_1([z]),e^{i\theta_2}(\varphi_\mu)_2([z]))
  \end{equation}
  for all $[z]\in\Omega_\mu$.
\end{lem}

\begin{rem}
  Observe that locally, \eqref{eq:varphi_defn} implies that $$\tilde\varphi(z)=\left(\left(\frac{z_1}{2}\right)^{\frac{1}{\mu}},e^{\frac{\pi}{2}}z_1^{\frac{i}{2}}z_2\right)$$ for some local holomorphic branches of the appropriate power functions.  However, the definition given in \eqref{eq:varphi_defn} is essential to capture the dependence of this branch on $|z_2|$.
\end{rem}

\begin{proof}
  Let $\tilde\rho$ be the defining function for $\tilde\Omega$ given by \eqref{eq:tilde_rho_defn}.  Since $\tilde\rho(0,z_2)=0$ for all $z_2\in\mathbb{C}\backslash\{0\}$, $z_1\neq 0$ on $\tilde\Omega$, and hence $\log^{|z_2|}(z_1)$ is well-defined on $\tilde\Omega$.  We can rewrite $\tilde\rho$ in the form
  \begin{equation}
  \label{eq:tilde_rho_alternate}
    \tilde\rho(z)=|z_1|^2+2\re\left(z_1 e^{-i\log|z_2|^2}\right).
  \end{equation}
  For $z\in\tilde\Omega$, \eqref{eq:tilde_rho_alternate} gives us
  \begin{multline*}
    \tilde\rho(z)=\tilde\rho(e^{\log^{|z_2|}(z_1)},z_2)=\\
    e^{2\re\log^{|z_2|}(z_1)}+2e^{\re\log^{|z_2|}(z_1)}\cos(\im\log^{|z_2|}(z_1)-\log|z_2|^2).
  \end{multline*}
  Hence $\tilde\rho(z)<0$ if and only if
  \begin{align*}
    \frac{1}{2}e^{\re\log^{|z_2|}(z_1)}&<-\cos(\im\log^{|z_2|}(z_1)-\log|z_2|^2)\\
    &=\cos(\pi-\im\log^{|z_2|}(z_1)+\log|z_2|^2).
  \end{align*}
  Since $\abs{\pi-\im\log^{|z_2|}(z_1)+\log|z_2|^2}\leq\pi$ by \eqref{eq:log_branch}, we see that $z\in\tilde\Omega$ if and only if $\re\log^{|z_2|}(z_1)<\log 2$ and
  \[
    \abs{\pi-\im\log^{|z_2|}(z_1)+\log|z_2|^2}<\arccos\left(\frac{1}{2}e^{\re\log^{|z_2|}(z_1)}\right).
  \]
  Since $\re\log^{|z_2|}(z_1)=\log|z_1|$, we see that $z\in\tilde\Omega$ if and only if $0<|z_1|<2$ and
  \begin{equation}
  \label{eq:defining_function_interior}
    \abs{\pi-\im\log^{|z_2|}(z_1)+\log|z_2|^2}<\arccos\left(\frac{1}{2}|z_1|\right).
  \end{equation}
  In particular, \eqref{eq:defining_function_interior} guarantees that $\abs{\pi-\im\log^{|z_2|}(z_1)+\log|z_2|^2}<\frac{\pi}{2}$ on $\tilde\Omega$, so $\log^{|z_2|}(z_1)$ is a holomorphic function on $\tilde\Omega$ and $\log^{|z_2|}(z_1)$ is locally constant on $\tilde\Omega$ with respect to $z_2$.

  Fix $z\in\tilde\Omega$.  Then
  \[
    \abs{\tilde\varphi_1(z)}=\abs{\exp\left(\frac{\log^{|z_2|}(z_1)-\log 2}{\mu}\right)}=\exp\left(\frac{\log|z_1|-\log 2}{\mu}\right).
  \]
  If we solve for $|z_1|$ in terms of $|\tilde\varphi_1(z)|$, we obtain
  \[
    |z_1|=\exp\left(\mu\log\abs{\tilde\varphi_1(z)}+\log 2\right),
  \]
  so we have
  \begin{equation}
  \label{eq:abs_z_1}
    |z_1|=2\abs{\tilde\varphi_1(z)}^{\mu}.
  \end{equation}
  Next, we compute
  \[
    \log\abs{\tilde\varphi_2(z)}^2=\log\abs{z_2\exp\left(\frac{1}{2}\left(\pi+i\log^{|z_2|}(z_1)\right)\right)}^2,
  \]
  so
  \[
    \log\abs{\tilde\varphi_2(z)}^2=\log\abs{z_2}^2+\pi-\im\log^{|z_2|}(z_1),
  \]
  Since $\re\log^{|z_2|}(z_1)=\log|z_1|$, this implies
  \[
    \log|z_1|-i\log\abs{\tilde\varphi_2(z)}^2=-i\log\abs{z_2}^2-i\pi+\log^{|z_2|}(z_1),
  \]
  which we may exponentiate to obtain
  \[
    |z_1|e^{-i\log\abs{\tilde\varphi_2(z)}^2}=-z_1 e^{-i\log\abs{z_2}^2}.
  \]
  Dividing by $|z_1|$ gives us
  \begin{equation}
  \label{eq:e_i_log_abs_z_2_squared}
    \frac{z_1}{|z_1|}e^{-i\log|z_2|^2}=-e^{-i\log\abs{\tilde\varphi_2(z)}^2}.
  \end{equation}
  Substituting \eqref{eq:abs_z_1} and \eqref{eq:e_i_log_abs_z_2_squared} in \eqref{eq:tilde_rho_alternate}, we see that
  \begin{equation}
  \label{eq:defining_function_varphi}
    \tilde\rho(z)=4\abs{\tilde\varphi_1(z)}^{\mu}\left(\abs{\tilde\varphi_1(z)}^{\mu}-\cos\left(\log\abs{\tilde\varphi_2(z)}^2\right)\right).
  \end{equation}
  Since $z\in\tilde\Omega$ implies that $\tilde\rho(z)<0$, \eqref{eq:defining_function_varphi} guarantees that
  \[
    \cos\left(\log\abs{\tilde\varphi_2(z)}^2\right)>\abs{\tilde\varphi_1(z)}^\mu>0.
  \]
  Hence $1>|\tilde\varphi_1(z)|>0$ and $\abs{\log\abs{\tilde\varphi_2(z)}^2}<\arccos\left(\abs{\tilde\varphi_1(z)}^\mu\right)$, so $\tilde\varphi(z)\in D_\mu$.

  For the converse, we fix $w\in D_\mu$.  Let $\log w_1$ denote the principal logarithm of $w_1$.  Then $w_1=\tilde\varphi_1(z)$ if and only if
  \begin{equation}
  \label{eq:log_z_1_preimage}
    \log^{|z_2|}(z_1)=\mu\left(\log w_1+2\pi k i\right)+\log 2
  \end{equation}
  for some $k\in\mathbb{Z}$.  Exponentiating both sides of \eqref{eq:log_z_1_preimage} gives us
  \begin{equation}
  \label{eq:z_1_preimage}
    z_1=2\exp\left(\mu\log w_1+2\pi\mu k i\right),
  \end{equation}
  so
  \begin{equation}
  \label{eq:abs_z_1_preimage}
    |z_1|=2\exp\left(\mu\log|w_1|\right)=2|w_1|^{\mu}
  \end{equation}
  Since $0<|w_1|<1$, \eqref{eq:abs_z_1_preimage} implies $0<|z_1|<2$.  We also know that $w_2=\tilde\varphi_2(z)$ if and only if
  \[
    z_2=w_2\exp\left(-\frac{1}{2}\left(\pi+i\log^{|z_2|}(z_1)\right)\right).
  \]
  If $w=\tilde\varphi(z)$, we may substitute \eqref{eq:log_z_1_preimage} to obtain
  \[
    z_2=w_2\exp\left(-\frac{1}{2}\left(\pi+i\left(\mu\left(\log w_1+2\pi k i\right)+\log 2\right)\right)\right),
  \]
  and so
  \begin{equation}
  \label{eq:z_2_preimage}
    z_2=e^{\pi\mu k} w_2\exp\left(-\frac{1}{2}\left(\pi+i\left(\mu\log w_1+\log 2\right)\right)\right).
  \end{equation}
  We use \eqref{eq:z_2_preimage} to compute
  \begin{equation}
  \label{eq:log_z_2_preimage}
    \log|z_2|^2=2\pi\mu k+\log|w_2|^2-\pi+\mu\im\log w_1.
  \end{equation}
  Combining \eqref{eq:log_z_1_preimage} and \eqref{eq:log_z_2_preimage}, we obtain
  \[
    \pi-\im\log^{|z_2|}(z_1)+\log|z_2|^2=\log|w_2|^2,
  \]
  and since \eqref{eq:abs_z_1_preimage} implies
  \[
    \arccos\left(\frac{1}{2}|z_1|\right)=\arccos\left(|w_1|^{\mu}\right),
  \]
  we see that $z$ satisfies \eqref{eq:defining_function_interior}.  We have already shown that $0<|z_1|<2$, so $z\in\tilde\Omega$ whenever $w\in D_\mu$, and hence $\tilde\varphi$ is surjective.  Moreover, \eqref{eq:z_1_preimage} and \eqref{eq:z_2_preimage} imply that any two elements of the preimage of $w$ (parameterized by $k\in\mathbb{Z}$) are in the same equivalence class, so $\tilde\varphi$ induces a well-defined injection $\varphi_\mu$ on $\Omega_\mu$.  We conclude that $\varphi_\mu$ is a biholomorphism from $\Omega_\mu$ to $D_\mu$.

  Finally, we consider the isometry $\tilde\psi_{\theta_1,\theta_2}$.  For $z\in\tilde\Omega$, we have
  \[
    \log^{|(\tilde\psi_{\theta_1,\theta_2})_2(z)|}((\tilde\psi_{\theta_1,\theta_2})_1(z))=\log^{|z_2|}(z_1)+i\mu\theta_1+2\pi\ell i
  \]
  for some $\ell\in\mathbb{Z}$.  Since
  \[
    \log|(\tilde\psi_{\theta_1,\theta_2})_2(z)|^2=\mu\theta_1+\log|z_2|^2,
  \]
  \eqref{eq:log_branch} implies that $\ell=0$, so
  \[
    \log^{|(\tilde\psi_{\theta_1,\theta_2})_2(z)|}((\tilde\psi_{\theta_1,\theta_2})_1(z))=\log^{|z_2|}(z_1)+i\mu\theta_1.
  \]
  With this identity, we may substitute directly in \eqref{eq:varphi_defn} to see \eqref{eq:isometry_push_forward}.

\end{proof}

\begin{cor}
\label{cor:hyperconvex}
  $\Omega_\mu$ is not hyperconvex.
\end{cor}

\begin{proof}
  Suppose that $f:\Omega_\mu\rightarrow\mathbb{R}$ is a bounded plurisubharmonic exhaustion function for $\Omega_\mu$.  Let $\varphi_\mu:\Omega_\mu\rightarrow D_\mu$ be the biholomorphic map induced by \eqref{eq:varphi_defn}.  Then $f\circ\varphi_\mu^{-1}$ is a bounded plurisubharmonic exhaustion function for $D_\mu$.  For $\zeta\in\mathbb{C}$, we set $g(\zeta)=f\circ\varphi_\mu^{-1}(\zeta,1)$, so that (using \eqref{eq:D_defn}), $g$ is a bounded subharmonic exhaustion function for the punctured unit-disc in $\mathbb{C}$, but no such function can exist.
\end{proof}

Corollary \ref{cor:hyperconvex} immediately implies that the Diederich-Forn{\ae}ss index of $\Omega_\mu$ is zero, but this is already known since $\partial\Omega_\mu$ contains a compact complex submanifold of positive dimension (see Proposition 5.7 in \cite{AdYu21}).

\section{Special Functions}

Before discussing integration on $\Omega_\mu$ with respect to the flat metric on $M_\mu$, we will need several special functions.  
These are defined by
\begin{equation}
\label{eq:alpha_definition_new}
  \alpha(x,y)=\int_0^1 t^{x-1}(1-t)^{y-1}dt\text{ for all }x,y>0
\end{equation}
and
\begin{equation}
\label{eq:beta_definition_new}
  \beta(x,y)=\int_{-\frac{\pi}{2}}^{\frac{\pi}{2}}(\cos t)^{x-1}e^{yt}dt\text{ for all }x>0\text{ and }y\in\mathbb{R}.
\end{equation}
The basic identities satisfied by $\alpha$ are given in the following lemma.
\begin{lem}
\label{lem:alpha_identities}
  For $x,y>0$, we have
  \begin{gather}
  \label{eq:alpha_symmetry}
    \alpha(x,y)=\alpha(y,x)\\
  \label{eq:alpha_induction_y}
    \alpha(x,y+1)=\frac{y}{x+y}\alpha(x,y)\text{, and}\\
  \label{eq:alpha_induction_x}
    \alpha(x+1,y)=\frac{x}{x+y}\alpha(x,y).
  \end{gather}
\end{lem}


\begin{proof}
  To see \eqref{eq:alpha_symmetry}, we use the substitution $u=1-t$ in \eqref{eq:alpha_definition_new}.

  We easily check that
  \begin{equation}
  \label{eq:alpha_decomposition}
    \alpha(x,y)=\alpha(x,y+1)+\alpha(x+1,y).
  \end{equation}
  Since
  \[
    \frac{\partial}{\partial t}(t^x(1-t)^y)= x t^{x-1}(1-t)^y-y t^x(1-t)^{y-1},
  \]
  we may integrate both sides of this from $0$ to $1$ and obtain
  \begin{equation}
  \label{eq:alpha_induction_new}
    x\alpha(x,y+1)=y\alpha(x+1,y).
  \end{equation}
  Now \eqref{eq:alpha_induction_new} implies that $\alpha(x+1,y)=\frac{x}{y}\alpha(x,y+1)$, so we may substitute this in \eqref{eq:alpha_decomposition} to obtain \eqref{eq:alpha_induction_y}.  If we apply \eqref{eq:alpha_symmetry} to \eqref{eq:alpha_induction_y}, we obtain \eqref{eq:alpha_induction_x}.

\end{proof}

The following estimate will be crucial.
\begin{lem}
\label{lem:alpha_normalized_bound_new}
  For $0\leq s<\frac{1}{2}$ and $x,y>2s$, we have
  \begin{equation}
  \label{eq:alpha_normalized_bound_new}
    \frac{\alpha(x-2s,y-2s)\alpha(x+2s,y+2s)}{(\alpha(x,y))^2}\leq\frac{xy}{(x-2s)(y-2s)}.
  \end{equation}
\end{lem}

\begin{proof}
  For $s=0$, \eqref{eq:alpha_normalized_bound_new} is both trivial and sharp.  Henceforth, we assume that $0<s<\frac{1}{2}$.

  For $x,y>0$, set $p=\frac{1}{1-2s}$ and $q=\frac{1}{2s}$, so that $p>1$, $q>1$, $\frac{1}{p}+\frac{1}{q}=1$, and $z+2s=\frac{z}{p}+\frac{z+1}{q}$ for any $z>0$.  We may use these exponents in H\"older's inequality to obtain
  \begin{equation}
  \label{eq:Holder_alpha_positive}
    \alpha(x+2s,y+2s)\leq(\alpha(x,y))^{1-2s}(\alpha(x+1,y+1))^{2s}.
  \end{equation}
  Since \eqref{eq:alpha_induction_y} and \eqref{eq:alpha_induction_x} imply that
  \begin{equation}
  \label{eq:alpha_induction_double}
    \alpha(x+1,y+1)=\frac{xy}{(x+y+1)(x+y)}\alpha(x,y),
  \end{equation}
  we may substitute \eqref{eq:alpha_induction_double} in \eqref{eq:Holder_alpha_positive} to obtain
  \begin{equation}
  \label{eq:alpha_positive_upper_bound}
    \alpha(x+2s,y+2s)\leq\alpha(x,y)\left(\frac{xy}{(x+y+1)(x+y)}\right)^{2s}.
  \end{equation}

  For $x,y>2s$, set $p=\frac{1}{2s}$ and $q=\frac{1}{1-2s}$, so that $p>1$, $q>1$, $\frac{1}{p}+\frac{1}{q}=1$, and $z+1-2s=\frac{z}{p}+\frac{z+1}{q}$ for any $z>0$.  We may use these exponents in H\"older's inequality to obtain
  \begin{equation}
  \label{eq:Holder_alpha_negative}
    \alpha(x+1-2s,y+1-2s)\leq(\alpha(x,y))^{2s}(\alpha(x+1,y+1))^{1-2s}.
  \end{equation}
  Since $x,y>2s$, we may apply \eqref{eq:alpha_induction_double} to both the left-hand and right-hand sides of \eqref{eq:Holder_alpha_negative} and obtain
  \begin{multline*}
    \alpha(x-2s,y-2s)\leq\\
    \alpha(x,y)\frac{(x+y-4s+1)(x+y-4s)}{(x-2s)(y-2s)}\left(\frac{xy}{(x+y+1)(x+y)}\right)^{1-2s}.
  \end{multline*}
  Since $\frac{(x+y-4s+1)(x+y-4s)}{(x+y+1)(x+y)}\leq 1$ when $x,y>2s$, this implies
  \begin{equation}
  \label{eq:alpha_negative_upper_bound}
    \alpha(x-2s,y-2s)\leq\alpha(x,y)\frac{xy}{(x-2s)(y-2s)}\left(\frac{xy}{(x+y+1)(x+y)}\right)^{-2s}.
  \end{equation}
  Multiplying \eqref{eq:alpha_positive_upper_bound} and \eqref{eq:alpha_negative_upper_bound}, we obtain \eqref{eq:alpha_normalized_bound_new}.

\end{proof}

The basic identity satisfied by $\beta$ is the following:
\begin{lem}
\label{lem:beta_identities}
  For $x>0$ and $y\in\mathbb{R}$, we have
  \begin{equation}
  \label{eq:beta_induction_new}
    \beta(x+2,y)=\frac{x(x+1)}{(x+1)^2+y^2}\beta(x,y).
  \end{equation}
\end{lem}


\begin{proof}
  Since
  \begin{multline*}
    \frac{\partial}{\partial t}\left(y(\cos t)^{x+1}e^{yt}+(x+1)(\cos t)^x(\sin t)e^{yt}\right)=\\
    y^2(\cos t)^{x+1} e^{yt}-x(x+1)(\cos t)^{x-1}(\sin^2 t)e^{yt}+(x+1)(\cos t)^{x+1} e^{yt},
  \end{multline*}
  we may substitute $\sin^2 t=1-\cos^2 t$ to obtain
  \begin{multline}
  \label{eq:beta_integration_by_parts}
    \frac{\partial}{\partial t}\left(y(\cos t)^{x+1}e^{yt}+(x+1)(\cos t)^x(\sin t)e^{yt}\right)=\\
    ((x+1)^2+y^2)(\cos t)^{x+1} e^{yt}-x(x+1)(\cos t)^{x-1}e^{yt}.
  \end{multline}
  Integrating \eqref{eq:beta_integration_by_parts} from $-\frac{\pi}{2}$ to $\frac{\pi}{2}$, we obtain
  \[
    0=((x+1)^2+y^2)\beta(x+2,y)-x(x+1)\beta(x,y),
  \]
  from which \eqref{eq:beta_induction_new} follows.

\end{proof}

The analog of Lemma \ref{lem:alpha_normalized_bound_new} for $\beta$ is given by the following:
\begin{lem}
\label{lem:beta_normalized_bound_new}
  For $0\leq s<\frac{1}{2}$, $x>4s$, and $y>0$, we have
  \begin{equation}
  \label{eq:beta_normalized_bound_new}
    \frac{\beta(x-4s,y)\beta(x+4s,y)}{(\beta(x,y))^2}\leq\frac{(1+4s)x}{x-4s}.
  \end{equation}
\end{lem}

\begin{proof}
  For $s=0$, \eqref{eq:beta_normalized_bound_new} is both trivial and sharp.  Henceforth, we assume that $0<s<\frac{1}{2}$.

  For $x,y>0$, set $p=\frac{1}{1-2s}$ and $q=\frac{1}{2s}$, so that $p>1$, $q>1$, $\frac{1}{p}+\frac{1}{q}=1$, and $x+4s=\frac{x}{p}+\frac{x+2}{q}$.  We may use these exponents in H\"older's inequality to obtain
  \begin{equation}
  \label{eq:Holder_beta_positive}
    \beta(x+4s,y)\leq(\beta(x,y))^{1-2s}(\beta(x+2,y))^{2s}.
  \end{equation}
  We may substitute \eqref{eq:beta_induction_new} in \eqref{eq:Holder_beta_positive} to obtain
  \begin{equation}
  \label{eq:beta_positive_upper_bound}
    \beta(x+4s,y)\leq\beta(x,y)\left(\frac{x(x+1)}{(x+1)^2+y^2}\right)^{2s}.
  \end{equation}

  For $x,y>0$, set $p=\frac{1}{2s}$ and $q=\frac{1}{1-2s}$, so that $p>1$, $q>1$, $\frac{1}{p}+\frac{1}{q}=1$, and $x+2-4s=\frac{x}{p}+\frac{x+2}{q}$.  We may use these exponents in H\"older's inequality to obtain
  \begin{equation}
  \label{eq:Holder_beta_negative}
    \beta(x+2-4s,y)\leq(\beta(x,y))^{2s}(\beta(x+2,y))^{1-2s}.
  \end{equation}
  If $x>4s$, we may apply \eqref{eq:beta_induction_new} to both the left-hand and right-hand sides of \eqref{eq:Holder_beta_negative} and obtain
  \[
    \beta(x-4s,y)\leq\beta(x,y)\frac{(x-4s+1)^2+y^2}{(x-4s)(x-4s+1)}\left(\frac{x(x+1)}{(x+1)^2+y^2}\right)^{1-2s}.
  \]
  Since $\frac{(x-4s+1)^2+y^2}{(x+1)^2+y^2}\leq 1$ and $\frac{x+1}{x-4s+1}\leq 1+4s$ when $x>4s>0$, this implies
  \begin{equation}
  \label{eq:beta_negative_upper_bound}
    \beta(x-4s,y)\leq\beta(x,y)\frac{(1+4s)x}{x-4s}\left(\frac{x(x+1)}{(x+1)^2+y^2}\right)^{-2s}.
  \end{equation}
  Multiplying \eqref{eq:beta_positive_upper_bound} and \eqref{eq:beta_negative_upper_bound}, we obtain \eqref{eq:beta_normalized_bound_new}.

\end{proof}

\section{Integration Formulas}

With $D_\mu$ defined by \eqref{eq:D_defn} and $\varphi_\mu$ defined by Lemma \ref{lem:biholomorphism}, we may now endow $D_\mu$ with the unique K\"ahler metric under which $\varphi_\mu$ is an isometry between $\Omega_\mu$ and $D_\mu$.  Observe that if we set $L_1^z=e^{i\log|z_2|^2}\frac{\partial}{\partial z_1}$ and $L_2^z=z_2\frac{\partial}{\partial z_2}$, then $\{L_1^z,L_2^z\}$ is an orthonormal basis for $T^{1,0}(\tilde M)$.  Moreover, $L_1^z$ and $L_2^z$ are both invariant under the push-forward induced by the map $(z_1,z_2)\mapsto(e^{2\pi\mu i}z_1,e^{\pi\mu} z_2)$, so they induce an orthonormal basis for $T^{1,0}(M_\mu)$.  If we define $\theta^1_z=e^{-i\log|z_2|^2}dz_1$ and $\theta^2_z=z_2^{-1}dz_2$, then $\{\theta^1_z,\theta^2_z\}$ is the basis for $\Lambda^{1,0}(\tilde M)$ that is dual to $\{L_1^z,L_2^z\}$, and hence it is also orthonormal.  Observe that both of the induced bases are smooth on $M_\mu$.

Recall that on $\tilde\Omega$, $\log^{|z_2|}(z_1)$ is a holomorphic branch of $\log z_1$ that is locally constant with respect to $|z_2|$, and hence
\[
  L_1^z\log^{|z_2|}(z_1)=z_1^{-1}e^{i\log|z_2|^2}\text{ and }L_2^z\log^{|z_2|}(z_1)=0.
\]
If we set $w=\tilde\varphi(z)$, then \eqref{eq:abs_z_1} and \eqref{eq:e_i_log_abs_z_2_squared} give us
\[
  L_1^z\log^{|z_2|}(z_1)=-\frac{1}{2|w_1|^{\mu}}e^{i\log|w_2|^2},
\]
and hence we may use \eqref{eq:varphi_defn} to compute the push-forwards
\begin{align*}
  \tilde\varphi_* L_1^z&=L_1^z\tilde\varphi_1(z)\frac{\partial}{\partial w_1}+L_1^z\tilde\varphi_2(z)\frac{\partial}{\partial w_2}\\
  &=-\frac{w_1}{2\mu|w_1|^\mu}e^{i\log|w_2|^2}\frac{\partial}{\partial w_1}-\frac{i w_2}{4|w_1|^\mu}e^{i\log|w_2|^2}\frac{\partial}{\partial w_2}
\end{align*}
and
\[
  \tilde\varphi_* L_2^z=L_2^z\tilde\varphi_1(z)\frac{\partial}{\partial w_1}+L_2^z\tilde\varphi_2(z)\frac{\partial}{\partial w_2}
  =w_2\frac{\partial}{\partial w_2}.
\]
If we define
\begin{align}
\label{eq:L_1_w_defn}  L_1^w&=-\frac{w_1}{2\mu|w_1|^\mu}e^{i\log|w_2|^2}\frac{\partial}{\partial w_1}-\frac{i w_2}{4|w_1|^\mu}e^{i\log|w_2|^2}\frac{\partial}{\partial w_2}\text{ and}\\
\label{eq:L_2_w_defn}  L_2^w&=w_2\frac{\partial}{\partial w_2},
\end{align}
then $\{L_1^w,L_2^w\}$ is an orthonormal basis for $T^{1,0}(D_\mu)$ in the metric obtained by pushing forward the K\"ahler metric on $T^{1,0}(\Omega_\mu)$.  The dual basis, which is necessarily orthonormal as well, is given by $\{\theta^1_w,\theta^2_w\}$, where
\begin{align}
\label{eq:theta_w_1_defn}  \theta^1_w&=-\frac{2\mu|w_1|^\mu}{w_1}e^{-i\log|w_2|^2}dw_1\text{ and}\\
\label{eq:theta_w_2_defn}  \theta^2_w&=-\frac{i\mu}{2w_1}dw_1+\frac{1}{w_2}dw_2.
\end{align}
We note, for future reference, the following consequence of \eqref{eq:theta_w_1_defn}:
\begin{equation}
\label{eq:theta_w_1_inverse}
  dw_1=-\frac{w_1}{2\mu|w_1|^\mu}e^{i\log|w_2|^2}\theta^1_w.
\end{equation}
With an orthonormal basis for $(1,0)$-forms, we may construct a K\"ahler form
\begin{multline*}
  \omega_{D_\mu}=i\mu^2\left(\frac{16|w_1|^{2\mu}+1}{8|w_1|^2}\right)dw_1\wedge d\bar w_1\\
  +\frac{\mu}{4w_1\bar w_2}dw_1\wedge d\bar w_2-\frac{\mu}{4\bar w_1 w_2}dw_2\wedge d\bar w_1+\frac{i}{2|w_2|^2}dw_2\wedge d\bar w_2
\end{multline*}
and a volume element
\[
  dV_{D_\mu}=\frac{\mu^2|w_1|^{2\mu-2}}{|w_2|^2}dw_1\wedge dw_2\wedge d\bar w_1\wedge d\bar w_2.
\]
If we let $dV_{z}=\frac{i}{2}dz\wedge d\bar z$ denote the Euclidean volume element on $\mathbb{C}$ with respect to the variable $z$, we may decompose this in the form
\begin{equation}
\label{eq:D_volume_element}
  dV_{D_\mu}=\frac{4\mu^2|w_1|^{2\mu-2}}{|w_2|^2}dV_{w_1}\wedge dV_{w_2}.
\end{equation}
From \eqref{eq:D_volume_element}, we see that computing the Bergman projection with respect to the push-forward metric is equivalent to computing the weighted Bergman projection on $\mathbb{C}^2$ with respect to the singular weight $w\mapsto\frac{4\mu^2|w_1|^{2\mu-2}}{|w_2|^2}$.  For $q\geq 1$, we define $L^q(D_\mu;dV_{D_\mu})$ to be the space of $L^q$ functions on $D_\mu$ computed with respect to the volume element \eqref{eq:D_volume_element}.

From \eqref{eq:D_defn}, we see that $w\in D_\mu$ if and only if $0<|w_1|<1$ and $a(|w_1|)<|w_2|<b(|w_1|)$, where $a$ and $b$ are defined by \begin{align}
  \label{eq:a_defn}a(r)&=\exp\left(-\frac{1}{2}\arccos\left(r^{\mu}\right)\right)\text{ for }0\leq r\leq 1\text{ and}\\
  \label{eq:b_defn}b(r)&=\exp\left(\frac{1}{2}\arccos\left(r^{\mu}\right)\right)\text{ for }0\leq r\leq 1.
\end{align}
Hence, if $f\in L^1(D_\mu;dV_{D_\mu})$, then an integral on $D_\mu$ with respect to the volume element defined by \eqref{eq:D_volume_element} can be computed via the iterated integral:
\begin{multline}
\label{eq:D_integration_formula}
  \int_{D_\mu} f(w) dV_{D_\mu}=\\
  \int_{0<|w_1|<1}\left(\int_{a(|w_1|)<|w_2|<b(|w_1|)}\frac{4\mu^2|w_1|^{2\mu-2}}{|w_2|^2}f(w)dV_{w_2}\right)dV_{w_1}.
\end{multline}
To fully exploit the structure of $D_\mu$ as a Reinhardt domain, we consider the case in which $f(w)=g(|w_1|,|w_2|)$.  Converting to polar coordinates in \eqref{eq:D_integration_formula}, we have
\begin{equation}
\label{eq:D_integration_formula_circular_symmetry}
  \int_{D_\mu} g(|w_1|,|w_2|) dV_{D_\mu}=
  16\pi^2\mu^2\int_0^1\int_{a(r_1)}^{b(r_1)}r_1^{2\mu-1}r_2^{-1}g(r_1,r_2)dr_2 dr_1,
\end{equation}
provided that
\begin{equation}
\label{eq:D_integration_formula_circular_symmetry_finiteness_condition}
  \int_0^1\int_{a(r_1)}^{b(r_1)}r_1^{2\mu-1}r_2^{-1}\abs{g(r_1,r_2)}dr_2 dr_1<\infty.
\end{equation}
A standard argument from one complex variables demonstrates that we can extend \eqref{eq:D_integration_formula_circular_symmetry} to
\begin{equation}
\label{eq:D_integration_formula_circular_symmetry_extended}
  \int_{D_\mu} g(|w_1|,|w_2|)w_1^j w_2^k dV_{D_\mu}=0\text{ if }j\neq 0\text{ or }k\neq 0.
\end{equation}
for all $j,k\in\mathbb{Z}$ provided that
\begin{equation}
\label{eq:D_integration_formula_circular_symmetry_extended_finiteness_condition}
  \int_0^1\int_{a(r_1)}^{b(r_1)}r_1^{j+2\mu-1}r_2^{k-1}\abs{g(r_1,r_2)}dr_2 dr_1<\infty.
\end{equation}

We will estimate Sobolev spaces on $\Omega_\mu$ by using $L^2$ spaces weighted by a power of the distance to the boundary.  In order to push these estimates forward to $D_\mu$, we will need the following:
\begin{lem}
\label{lem:delta_estimates}
  If $\delta:D_\mu\rightarrow\mathbb{R}$ satisfies
  \begin{equation}
  \label{eq:delta_defn}
    \delta(\varphi_\mu([z]))=\dist([z],\partial\Omega_\mu)\text{ for all }[z]\in\Omega_\mu,
  \end{equation}
  where $\varphi_\mu$ is defined by \eqref{eq:varphi_defn}, then
  \begin{equation}
  \label{eq:delta_radial_symmetry}
    \delta(w)=\delta(|w_1|,|w_2|)\text{ for all }w\in D_\mu
  \end{equation}
  and there exists a constant $C>1$ independent of $w$ such that
  \begin{equation}
  \label{eq:delta_lower_bound}
    \delta(w)\geq C^{-1}\abs{w_1}^{\mu}\left(\cos\left(\log\abs{w_2}^2\right)-\abs{w_1}^{\mu}\right)
  \end{equation}
  and
  \begin{equation}
  \label{eq:delta_upper_bound}
    \delta(w)\leq C\abs{w_1}^{\mu}\left(\cos\left(\log\abs{w_2}^2\right)-\abs{w_1}^{\mu}\right)
  \end{equation}
  for all $w\in D_\mu$.
\end{lem}

\begin{proof}
  Observe that $D_\mu$ is a Reinhardt domain and \eqref{eq:isometry_push_forward} implies that if we push-forward the K\"ahler metric on $\Omega_\mu$ to $D_\mu$ by the map $\varphi_\mu$, rotations in each variable are isometries on $D_\mu$.  Hence, the distance to the boundary of $D_\mu$ must satisfy the symmetry \eqref{eq:delta_radial_symmetry}.

  For $w\in D_\mu$, define
  \[
    \delta_0(w)=4\abs{w_1}^{\mu}\left(\cos\left(\log\abs{w_2}^2\right)-\abs{w_1}^{\mu}\right)
  \]
  Let $\tilde\rho$ denote the defining function for $\tilde\Omega$ given by \eqref{eq:tilde_rho_defn}.  By \eqref{eq:defining_function_varphi}, $\delta_0\left(\tilde\varphi(z)\right)=-\tilde\rho(z)$ for all $z\in\tilde\Omega$.  Since the absolute value of any defining function is comparable to the distance to the boundary, $\delta_0(\varphi_\mu([z]))$ is comparable to $\dist([z],\partial\Omega_\mu)$.  Hence, $\delta_0(w)$ is comparable to $\delta(w)$ on $D_\mu$, from which \eqref{eq:delta_lower_bound} and \eqref{eq:delta_upper_bound} follow.

\end{proof}

As usual on Reinhardt domains (and as a result of \eqref{eq:D_integration_formula_circular_symmetry_extended}), we will see that the Bergman space on $D_\mu$ is spanned by holomorphic monomials.  We will see that the Sobolev norms of these monomials can be estimated by considering the $L^2$ norms weighted by a power of the distance to the boundary.  The function $\lambda$ defined below by \eqref{eq:lambda_defn} below will give us the key tool to estimate these norms.
\begin{lem}
  For $x,y,s\in\mathbb{R}$, and $\delta$ defined by \eqref{eq:delta_defn}, if we define
  \begin{equation}
  \label{eq:lambda_defn}
    \lambda(x,y,s)=\int_{D_\mu}|w_1|^{2x}|w_2|^{2y}(\delta(w))^{-2s}dV_{D_\mu},
  \end{equation}
  then we have
  \begin{equation}
  \label{eq:lambda_integrability}
    \lambda(x,y,s)<\infty\text{ if and only if }\frac{x}{\mu}+1-s>0\text{ and }s<\frac{1}{2}.
  \end{equation}
  If $0\leq s<\frac{1}{2}$, we have
  \begin{equation}
  \label{eq:lambda_normalized_bounds}
    \frac{\lambda(x,y,s)\lambda(x,y,-s)}{(\lambda(x,y,0))^2}\leq C^{8|s|}\frac{\left(\frac{2x}{\mu}+2\right)(1+4s)\left(\frac{2x}{\mu}+3\right)}{\left(\frac{2x}{\mu}+2-2s\right)(1-2s)\left(\frac{2x}{\mu}+3-4s\right)}
  \end{equation}
  for the constant $C>1$ given by Lemma \ref{lem:delta_estimates}.
\end{lem}

\begin{proof}
  Since the integrand in \eqref{eq:lambda_defn} is positive on $D_\mu$, \eqref{eq:D_integration_formula_circular_symmetry_finiteness_condition} will be satisfied precisely when \eqref{eq:D_integration_formula_circular_symmetry} is finite, so we will not need to carefully distinguish the iterated integral from the integral with respect to the higher dimensional measure.  By \eqref{eq:D_integration_formula_circular_symmetry} and \eqref{eq:delta_radial_symmetry}, we have
  \[
    \lambda(x,y,s)=16\pi^2\mu^2\int_0^1\int_{a(r_1)}^{b(r_1)}r_1^{2x+2\mu-1}r_2^{2y-1}(\delta(r_1,r_2))^{-2s}dr_2 dr_1.
  \]
  When $\lambda(x,y,s)$ is finite, we may use \eqref{eq:a_defn} and \eqref{eq:b_defn} to change the order of integration and obtain
  \[
    \lambda(x,y,s)=
    16\pi^2\mu^2\int_{e^{-\frac{\pi}{4}}}^{e^{\frac{\pi}{4}}}\int_0^{(\cos(\log r_2^2))^{1/\mu}}r_1^{2x+2\mu-1}r_2^{2y-1}(\delta(r_1,r_2))^{-2s}dr_1 dr_2.
  \]
  Now \eqref{eq:delta_lower_bound} gives us
  \[
    \lambda(x,y,s)\leq
    16\pi^2\mu^2 C^{2|s|}\int_{e^{-\frac{\pi}{4}}}^{e^{\frac{\pi}{4}}}\int_0^{(\cos(\log r_2^2))^{1/\mu}}\frac{r_1^{2x+2\mu-2s\mu-1}r_2^{2y-1}}{\left(\cos\left(\log r_2^2\right)-r_1^{\mu}\right)^{2s}}dr_1 dr_2
  \]
  and \eqref{eq:delta_upper_bound} gives us
  \[
    \lambda(x,y,s)\geq
    \frac{16\pi^2\mu^2}{C^{2|s|}}\int_{e^{-\frac{\pi}{4}}}^{e^{\frac{\pi}{4}}}\int_0^{(\cos(\log r_2^2))^{1/\mu}}\frac{r_1^{2x+2\mu-2s\mu-1}r_2^{2y-1}}{\left(\cos\left(\log r_2^2\right)-r_1^{\mu}\right)^{2s}}dr_1 dr_2.
  \]
  If we make the change of variables $u(r)=\left(\frac{r_1^\mu}{\cos\left(\log r_2^2\right)},\log r_2^2\right)$, then we have $r(u)=\left(\left(u_1\cos u_2\right)^{1/\mu},e^{\frac{1}{2}u_2}\right)$, so
  \[
    dr_1\wedge dr_2=\left(\frac{1}{\mu}u_1^{\frac{1-\mu}{\mu}}(\cos u_2)^{\frac{1}{\mu}}\right)\left(\frac{1}{2}e^{\frac{1}{2}u_2}\right)du_1\wedge du_2.
  \]
  Hence,
  \begin{multline*}
    \lambda(x,y,s)\leq\\
    8\pi^2\mu C^{2|s|}\int_{-\frac{\pi}{2}}^{\frac{\pi}{2}}\int_0^1 u_1^{\frac{2x}{\mu}+1-2s}\left(\cos u_2\right)^{\frac{2x}{\mu}+2-4s}e^{y u_2}(1-u_1)^{-2s}du_1 du_2
  \end{multline*}
  and
  \[
    \lambda(x,y,s)\geq
    \frac{8\pi^2\mu}{C^{2|s|}}\int_{-\frac{\pi}{2}}^{\frac{\pi}{2}}\int_0^1 u_1^{\frac{2x}{\mu}+1-2s}\left(\cos u_2\right)^{\frac{2x}{\mu}+2-4s}e^{y u_2}(1-u_1)^{-2s}du_1 du_2
  \]
  Since the integrand is separable, we may decompose these into the product of two integrals in one variable using \eqref{eq:alpha_definition_new} and \eqref{eq:beta_definition_new}, from which
  \begin{multline}
  \label{eq:lambda_bounds}
    \frac{8\pi^2\mu}{C^{2|s|}}\alpha\left(\frac{2x}{\mu}+2-2s,1-2s\right)\beta\left(\frac{2x}{\mu}+3-4s,y\right)\leq\lambda(x,y,s)\\
    \leq 8\pi^2\mu C^{2|s|}\alpha\left(\frac{2x}{\mu}+2-2s,1-2s\right)\beta\left(\frac{2x}{\mu}+3-4s,y\right)
  \end{multline}
  follows.  We know that $\alpha\left(\frac{2x}{\mu}+2-2s,1-2s\right)$ is integrable if and only if $\frac{2x}{\mu}+2-2s>0$ and $1-2s>0$, $\beta\left(\frac{2x}{\mu}+3-4s,y\right)$ is integrable if and only if $\frac{2x}{\mu}+3-4s>0$, and
  \[
    \frac{2x}{\mu}+3-4s=\left(\frac{2x}{\mu}+2-2s\right)+(1-2s),
  \]
  so \eqref{eq:lambda_integrability} follows.  We may combine \eqref{eq:alpha_normalized_bound_new} and \eqref{eq:beta_normalized_bound_new} with \eqref{eq:lambda_bounds} to obtain \eqref{eq:lambda_normalized_bounds}.
\end{proof}

\section{Sobolev spaces}

Since $\Omega_\mu$ is a relatively compact domain with smooth boundary, Sobolev spaces are easily defined even for non-integer orders.  Let
\begin{align*}
  \tilde U_1&=\{z\in\tilde M:1<|z_2|<e^{\pi\mu}\}\text{ and}\\
  \tilde U_2&=\{z\in\tilde M:e^{-\frac{\pi}{2}\mu}<|z_2|<e^{\frac{\pi}{2}\mu}\}.
\end{align*}
For $j\in\{1,2\}$, $\tilde U_j$ may be identified with a domain in $\mathbb{C}^2$, and since the metric on $\tilde M$ is equivalent to the Euclidean metric on $\mathbb{C}^2$ on a neighborhood of $\tilde U_j$, we may define the $L^2$ Sobolev space $W^s(\tilde U_j\cap\tilde\Omega)$ as the restrictions of elements in $W^s(\mathbb{C}^2)$ for any $s\in\mathbb{R}$.  Since \eqref{eq:nonvanishing_first_derivative} implies that each $\tilde U_j\cap\tilde\Omega$ is a Lipschitz domain, Theorem 1.4.3.1 in \cite{Gri85} implies that this definition will agree with other common definitions.  If we set $U_j^\mu=\{[z]\in M_\mu:z\in\tilde U_j\}$ for $j\in\{1,2\}$, then $\{U_j^\mu\}_{j\in\{1,2\}}$ is an open cover for $\Omega_\mu$.  For $j\in\{1,2\}$, $U_j^\mu$ is naturally biholomorphic to $\tilde U_j$, so we may define $W^s(U_j^\mu\cap\Omega_\mu)$ for any $s\in\mathbb{R}$.  By a partition of unity, we may now define $W^s(\Omega_\mu)$ for any $s\in\mathbb{R}$.

On the other hand, $D_\mu$ does not have a Lipschitz boundary, and the metric for $D_\mu$ does not extend smoothly to a neighborhood of $\overline{D_\mu}$, so defining the Sobolev spaces intrinsically is problematic.  Instead, we define $W^s(D_\mu;dV_{D_\mu})$ to be the set of all measurable functions $f$ on $D_\mu$ such that $f\circ\varphi_\mu\in W^s(\Omega_\mu)$, where $\varphi_\mu$ is the biholomorphism induced by \eqref{eq:varphi_defn}.  Since the metric on $D_\mu$ is obtained by pushing forward the metric on $\Omega_\mu$ via $\varphi_\mu$, $W^s(D_\mu;dV_{D_\mu})$ has the usual definition when $s$ is an integer, provided that we take derivatives with respect to the orthonormal coordinates given by \eqref{eq:L_1_w_defn} and \eqref{eq:L_2_w_defn}.

For $s<\frac{1}{2}$ and $\delta$ defined by \eqref{eq:delta_defn}, let $L^2(D_\mu;\delta^{-2s}dV_{D_\mu})$ denote the weighted $L^2$ space on $D_\mu$ with inner product
\[
  \left<f,g\right>_{L^2(D_\mu;\delta^{-2s}dV_{D_\mu})}=\int_{D_\mu}f\bar g\delta^{-2s}dV_{D_\mu}.
\]
For $s<\frac{3}{2}$ and $\{L_1^w,L_2^w\}$ defined by \eqref{eq:L_1_w_defn} and \eqref{eq:L_2_w_defn}, let $W^1(D_\mu;\delta^{2-2s}dV_{D_\mu})$ denote the weighted Sobolev space on $D_\mu$ with inner product
\[
  \left<f,g\right>_{W^1(D_\mu;\delta^{2-2s}dV_{D_\mu})}=\int_{D_\mu}f\bar gdV_{D_\mu}+\sum_{\ell=1}^2\int_{D_\mu}(L_\ell^w f)\overline{L_\ell^w g}\delta^{2-2s}dV_{D_\mu}.
\]
We have the following:
\begin{lem}
\label{lem:Sobolev_imbeddings}
  If $0\leq s<\frac{1}{2}$, then
  \begin{equation}
  \label{eq:Sobolev_imbeddings}
    W^1(D_\mu;\delta^{2-2s}dV_{D_\mu})\subset W^s(D_\mu;dV_{D_\mu})\subset L^2(D_\mu;\delta^{-2s}dV_{D_\mu}).
  \end{equation}
\end{lem}

\begin{proof}
  On $\tilde U_j\cap\tilde\Omega$ for $j\in\{1,2\}$, the first inclusion follows from Theorem C.2 in \cite{ChSh01} and the second inclusion follows from Theorem 1.4.4.3 in \cite{Gri85}.  On $\tilde U_j\cap\tilde\Omega$, we have
  \[
    \left(\dist\left(z,\partial\left(\tilde U_j\cap\tilde\Omega\right)\right)\right)^{2-2s}\leq\left(\dist\left(z,\partial\tilde\Omega\right)\right)^{2-2s}
  \]
  and
  \[
    \left(\dist\left(z,\partial\left(\tilde U_j\cap\tilde\Omega\right)\right)\right)^{-2s}\geq\left(\dist\left(z,\partial\tilde\Omega\right)\right)^{-2s},
  \]
  so we may use a partition of unity to prove \eqref{eq:Sobolev_imbeddings} on $\Omega_\mu$.  Pushing forward the result along $\varphi_\mu$, we obtain \eqref{eq:Sobolev_imbeddings}.
\end{proof}

The following is well-known for relatively compact domains in $\mathbb{C}^n$ with Lipschitz boundaries, since holomorphic functions have harmonic components.  On our non-Lipschitz domain with the metric that is singular on the boundary, we are still able to confirm this by pulling back to $\Omega_\mu$.
\begin{lem}
  For $0\leq s<\frac{1}{2}$ and $\delta$ defined by \eqref{eq:delta_defn},
  \begin{multline}
  \label{eq:Sobolev_embeddings_holomorphic}
    \mathcal{O}(D_\mu)\cap W^1(D_\mu;\delta^{2-2s}dV_{D_\mu})=\mathcal{O}(D_\mu)\cap W^s(D_\mu;dV_{D_\mu})=\\
    \mathcal{O}(D_\mu)\cap L^2(D_\mu;\delta^{-2s}dV_{D_\mu}).
  \end{multline}
\end{lem}

\begin{proof}
  Since holomorphic functions have harmonic coefficients,
  \[
    \mathcal{O}(\Omega_\mu)\cap L^2(\Omega_\mu;\delta^{-2s}dV_{\Omega_\mu})\subset\mathcal{O}(\Omega_\mu)\cap W^1(\Omega_\mu;\delta^{2-2s}dV_{\Omega_\mu})
  \]
  follows from Lemme 1 in \cite{Det81}.  If we push this inclusion forward along $\varphi_\mu$, we obtain the corresponding result for $D_\mu$.  Combined with \eqref{eq:Sobolev_imbeddings}, we obtain \eqref{eq:Sobolev_embeddings_holomorphic}.
\end{proof}

Since $\{\theta^1_z,\theta^2_z\}$ is smooth in a neighborhood of $\overline{\Omega_\mu}$, we may define $W^s_{1,0}(\Omega_\mu)$ to be linear combinations of $\{\theta^1_z,\theta^2_z\}$ with coefficients in $W^s(\Omega_\mu)$, and $W^s_{2,0}(\Omega_\mu)$ to be multiples of $\theta^1_z\wedge\theta^2_z$ by functions in $W^s(\Omega_\mu)$.  Pushing forward these spaces to $D_\mu$, we see that we may define $W^s_{1,0}(D_\mu;dV_{D_\mu})$ (resp. $W^s_{2,0}(D_\mu;dV_{D_\mu})$) to be linear combinations of $\theta^1_w$ and $\theta^2_w$ (resp. $\theta^1_w\wedge\theta^2_w$) with coefficients in $W^s(D_\mu;dV_{D_\mu})$ as defined above.

\section{Bases for Bergman Spaces}

Now we have the tools to compute an orthonormal basis for the weighted Bergman spaces $A^2_{p,0}(D_\mu;\delta^{-2s}dV_{D_\mu}):=\mathcal{O}(D_\mu)\cap L^2_{p,0}(D_\mu;\delta^{-2s}dV_{D_\mu})$, where $0\leq s<\frac{1}{2}$ and $p\in\{0,1,2\}$.  Since we have chosen our metric so that $D_\mu$ is isometric to $\Omega_\mu$, this will give us an orthonormal basis for the Bergman space $A^2_{p,0}(\Omega_\mu;\delta^{-2s}dV_{\Omega_\mu})$.
\begin{lem}
\label{lem:orthonormal_basis}
  For $s<\frac{1}{2}$ and $\delta$ defined by \eqref{eq:delta_defn}, an orthonormal basis for
  \[
    A^2(D_\mu;\delta^{-2s}dV_{D_\mu})
  \]
  is given by
  \begin{equation}
  \label{eq:orthonormal_basis}
    \set{(\lambda(j,k,s))^{-1/2} w_1^j w_2^k:j,k\in\mathbb{Z},j>(s-1)\mu},
  \end{equation}
  where $\lambda(j,k,s)$ is defined by \eqref{eq:lambda_defn}, an orthonormal basis for
  \[
    A^2_{1,0}(D_\mu;\delta^{-2s}dV_{D_\mu})
  \]
  is given by
  \begin{multline}
  \label{eq:orthonormal_basis_one_form}
    \set{(\lambda(j-\mu,k,s))^{-1/2} 2\mu w_1^{j-1} w_2^k dw_1:j,k\in\mathbb{Z},j>s\mu}\cup\\
    \set{(\lambda(j,k,s))^{-1/2} w_1^j w_2^k\theta^2_w:j,k\in\mathbb{Z},j>(s-1)\mu},
  \end{multline}
  where $\theta^2_w$ is given by \eqref{eq:theta_w_2_defn}, and an orthonormal basis for
  \[
    A^2_{2,0}(D_\mu;\delta^{-2s}dV_{D_\mu})
  \]
  is given by
  \begin{equation}
  \label{eq:orthonormal_basis_two_form}
    \set{(\lambda(j-\mu,k,s))^{-1/2} 2\mu w_1^{j-1} w_2^k dw_1\wedge\theta^2_w:j,k\in\mathbb{Z},j>s\mu}.
  \end{equation}
\end{lem}

\begin{proof}
  Since $D_\mu$ is a Reinhardt domain, every holomorphic function on $D_\mu$ can be written as a convergent Laurent series in $w_1$ and $w_2$, i.e., if $f\in\mathcal{O}(D_\mu)$, then there exist coefficients $\{d_{j,k}\}_{j,k\in\mathbb{Z}}\subset\mathbb{C}$ such that
  \begin{equation}
  \label{eq:Laurent_series_expansion}
    f(w)=\sum_{j,k\in\mathbb{Z}} d_{j,k} w_1^j w_2^k
  \end{equation}
  on $D_\mu$.  To postpone addressing the integrability condition in \eqref{eq:D_integration_formula_circular_symmetry_extended_finiteness_condition}, we define
  \[
    D_\mu^\epsilon=\set{w\in\mathbb{C}^2:\epsilon<|w_1|<1\text{ and }\abs{\log|w_2|^2}<\arccos\left(|w_1|^{\mu}\right)}
  \]
  for any $0<\epsilon<1$.  Since $D_\mu^\epsilon\subset D_\mu$, \eqref{eq:Laurent_series_expansion} also holds on $D_\mu^\epsilon$.  Observe that $w\in D_\mu^\epsilon$ if and only if $\epsilon<|w_1|<1$ and $a(|w_1|)<|w_2|<b(|w_1|)$, where $a$ and $b$ are defined by \eqref{eq:a_defn} and \eqref{eq:b_defn}.  For $j,k,\ell,m\in\mathbb{Z}$, \eqref{eq:delta_radial_symmetry} implies that
  \[
    w_1^jw_2^k\bar w_1^\ell\bar w_2^m(\delta(w))^{-2s}=w_1^{j-\ell}w_2^{k-m}|w_1|^{2\ell}|w_2|^{2m}(\delta(|w_1|,|w_2|))^{-2s}
  \]
  so \eqref{eq:D_integration_formula_circular_symmetry_extended} implies that $w_1^j w_2^k$ and $w_1^\ell w_2^m$ are orthogonal in $$L^2(D_\mu^\epsilon;\delta^{-2s}dV_{D_\mu})$$ whenever $j\neq\ell$ or $k\neq m$.  Hence, $\{w_1^j w_2^k\}_{j,k\in\mathbb{Z}}$ is an orthogonal basis for
  \[
    A^2(D_\mu^\epsilon;\delta^{-2s}dV_{D_\mu}).
  \]
  For $f$ given by \eqref{eq:Laurent_series_expansion}, we have
  \begin{equation}
    \norm{f}^2_{L^2(D_\mu^\epsilon;\delta^{-2s}dV_{D_\mu})}=\sum_{j,k\in\mathbb{Z}}|d_{j,k}|^2\int_{D_\mu^\epsilon}|w_1|^{2j}|w_2|^{2k}(\delta(w))^{-2s}dV_{D_\mu}.
  \end{equation}
  The Monotone Convergence Theorem implies that
  \[
    \norm{f}^2_{L^2(D_\mu^\epsilon;\delta^{-2s}dV_{D_\mu})}\rightarrow\norm{f}^2_{L^2(D_\mu;\delta^{-2s}dV_{D_\mu})}
  \]
  as $\epsilon\rightarrow 0^+$, so if $f\in L^2(D_\mu;\delta^{-2s}dV_{D_\mu})$, then for each $j,k\in\mathbb{Z}$ either $d_{j,k}=0$ or
  \[
    \int_{D_\mu}|w_1|^{2j}|w_2|^{2k}(\delta(w))^{-2s}dV_{D_\mu}<\infty.
  \]
  By \eqref{eq:lambda_integrability}, we see that $d_{j,k}=0$ whenever $\frac{j}{\mu}+1-s\leq 0$, or equivalently $j\leq(s-1)\mu$.  Since \eqref{eq:lambda_defn} gives us the tool needed to normalize our basis when $j>(s-1)\mu$, \eqref{eq:orthonormal_basis} is an orthonormal basis for $A^2(D_\mu;\delta^{-2s}dV_{D_\mu})$.
  
  Since $\theta^2_w$ is holomorphic and unit-length, the second set in \eqref{eq:orthonormal_basis_one_form} is an orthonormal basis for the subspace of $A^2_{1,0}(D_\mu;\delta^{-2s}dV_{D_\mu})$ spanned by $\theta^2_w$ by the same computations used to derive \eqref{eq:orthonormal_basis}.  For $j,k\in\mathbb{Z}$, we may use \eqref{eq:theta_w_1_inverse} to compute
  \begin{align*}
    \norm{2\mu w_1^{j-1}w_2^kdw_1}^2_{L^2(D_\mu;\delta^{-2s}dV_{D_\mu})}&=\int_{D_\mu}|w_1|^{2(j-\mu)}|w_2|^{2k}(\delta(w))^{-2s}dV_{D_\mu}\\
    &=\lambda(j-\mu,k,s),
  \end{align*}
  and by \eqref{eq:lambda_integrability}, we see that this is finite if and only if $\frac{j}{\mu}-s>0$.  We may now proceed as before to show that the first set in \eqref{eq:orthonormal_basis_one_form} completes an orthonormal basis for $A^2_{1,0}(D_\mu;\delta^{-2s}dV_{D_\mu})$.
  
  By taking the wedge product of each basis element in \eqref{eq:orthonormal_basis_one_form} with the holomorphic unit-length one-form $\theta^2_w$ that is orthogonal to $dw_1$, we obtain an orthonormal basis for $A^2_{2,0}(D_\mu;\delta^{-2s}dV_{D_\mu})$, given by \eqref{eq:orthonormal_basis_two_form}.


\end{proof}

\section{The Bergman Projection}

For $p\in\{0,1,2\}$, let $P_{\Omega_\mu}^p:L^2_{p,0}(\Omega_\mu)\rightarrow A^2_{p,0}(\Omega_\mu)$ denote the Bergman projection on $\Omega_\mu$ with respect to the metric on $M_\mu$.  Let $P_{D_\mu}^p:L^2_{p,0}(D_\mu;dV_{D_\mu})\rightarrow A^2_{p,0}(D_\mu;dV_{D_\mu})$ denote the Bergman projection on $D_\mu$ with respect to the metric obtained by pushing forward the metric on $\Omega_\mu$.  Suppose $0\leq s<\frac{1}{2}$.  If $f\in W^s_{p,0}(\Omega_\mu)$, then we have defined $W^s_{p,0}(D_\mu;dV_{D_\mu})$ so that there must exist $\tilde f\in W^s_{p,0}(D_\mu;dV_{D_\mu})$ satisfying $\varphi_\mu^*\tilde f=f$ on $\Omega_\mu$, where $\varphi_\mu$ is the biholomorphism induced by \eqref{eq:varphi_defn}.  By construction, $P_{D_\mu}^p\tilde f(\varphi_\mu([z]))=P_{\Omega_\mu}^pf([z])$.  Hence, we have
\begin{lem}
\label{lem:Sobolev_regularity_equivalence}
  For $0\leq s<\frac{1}{2}$ and $0\leq p\leq 2$, $P_{\Omega_\mu}^p$ is continuous in $W^s_{p,0}(\Omega_\mu)$ if and only if $P_{D_\mu}^p$ is continuous in $W^s_{p,0}(D_\mu;dV_{D_\mu})$.
\end{lem}

Now we have the tools to easily prove our main regularity results.
\begin{thm}
\label{thm:regularity}
  For
  \begin{equation}
  \label{eq:s_range_positive}
    0\leq s<\min\set{\frac{1}{2},\frac{1-\lfloor\mu\rfloor}{\mu}+1},
  \end{equation}
  the Bergman Projection $P_{\Omega_\mu}^0$ is continuous in $W^s(\Omega_\mu)$.  For
  \begin{equation}
  \label{eq:s_range_positive_2}
    0\leq s<\min\set{\frac{1}{2},\frac{1}{\mu}},
  \end{equation}
  the Bergman Projection $P_{\Omega_\mu}^1$ is continuous in $W^s_{1,0}(\Omega_\mu)$ and the Bergman Projection $P_{\Omega_\mu}^2$ is continuous in $W^s_{2,0}(\Omega_\mu)$.
\end{thm}

\begin{proof}
  Suppose that $s$ satisfies \eqref{eq:s_range_positive}.  Let $f\in W^s(D_\mu;dV_{D_\mu})$ and let $Z_0=\{(j,k)\in\mathbb{Z}^2:j>-\mu\}$.  By Lemma \ref{lem:orthonormal_basis}, $\set{\frac{w_1^j w_2^k}{\sqrt{\lambda(j,k,0)}}}_{(j,k)\in Z_0}$ is an orthonormal basis for $A^2(D_\mu;dV_{D_\mu})$, so
  \[
    P_{D_\mu}^0f(w)=\sum_{(j,k)\in Z_0}\int_{D_\mu} f(u)\frac{\bar u_1^j\bar u_2^k}{\sqrt{\lambda(j,k,0)}}dV_{D_\mu}\frac{w_1^j w_2^k}{\sqrt{\lambda(j,k,0)}}.
  \]
  If $j$ is an integer and $j>-\mu$, then $j\geq 1-\lfloor\mu\rfloor$.  By \eqref{eq:s_range_positive}, $j>\mu(s-1)$.  Hence, Lemma \ref{lem:orthonormal_basis} implies that for $s$ satisfying \eqref{eq:s_range_positive} $\set{\frac{w_1^j w_2^k}{\sqrt{\lambda(j,k,s)}}}_{(j,k)\in Z_0}$ is an orthonormal basis for $A^2(D_\mu;\delta^{-2s}dV_{D_\mu})$, and so
  \[
    \norm{P_{D_\mu}^0f}^2_{L^2(D_\mu;\delta^{-2s}dV_{D_\mu})}=\sum_{(j,k)\in Z_0}\abs{\int_{D_\mu}f(w)\frac{\bar w_1^j\bar w_2^k}{\sqrt{\lambda(j,k,0)}}dV_{D_\mu}}^2\frac{\lambda(j,k,s)}{\lambda(j,k,0)}.
  \]
  Since $j>\mu(s-1)$ implies that $\frac{2j}{\mu}+2-2s$ is uniformly bounded away from zero when $j$ is an integer, \eqref{eq:lambda_normalized_bounds} implies that there exists a constant $C_{\mu,s}>1$ such that
  \[
    \frac{\lambda(j,k,s)\lambda(j,k,-s)}{(\lambda(j,k,0))^2}\leq C_{\mu,s}
  \]
  for all $(j,k)\in Z_0$, and hence
  \[
    \norm{P_{D_\mu}^0f}^2_{L^2(D_\mu;\delta^{-2s}dV_{D_\mu})}\leq C_{\mu,s}\sum_{(j,k)\in Z_0}\abs{\int_{D_\mu}f(w)\frac{\bar w_1^j\bar w_2^k}{\sqrt{\lambda(j,k,-s)}}dV_{D_\mu}}^2.
  \]
  By \eqref{eq:lambda_defn} and \eqref{eq:D_integration_formula_circular_symmetry_extended}, $\set{\frac{\delta^{s}(w)w_1^j w_2^k}{\sqrt{\lambda(j,k,-s)}}}_{(j,k)\in Z_0}$ is an orthonormal set, so we have
  \[
    \norm{P_{D_\mu}^0f}^2_{L^2(D_\mu;\delta^{-2s}dV_{D_\mu})}\leq C_{\mu,s}\int_{D_\mu}|f(w)|^2(\delta(w))^{-2s}dV_{D_\mu}.
  \]
  Using \eqref{eq:Sobolev_imbeddings} for the upper bound and \eqref{eq:Sobolev_embeddings_holomorphic} for the lower bound, we have shown that $P_{D_\mu}^0$ is continuous in $W^s(D_\mu;dV_{D_\mu})$ whenever \eqref{eq:s_range_positive} holds, so the $p=0$ case follows from Lemma \ref{lem:Sobolev_regularity_equivalence}.
  
  Suppose that $s$ satisfies \eqref{eq:s_range_positive_2}.  Let $f\in W^s(D_\mu;dV_{D_\mu})$ and let $Z_1=\{(j,k)\in\mathbb{Z}^2:j>0\}$.  By Lemma \ref{lem:orthonormal_basis}, $\set{\frac{2\mu w_1^{j-1} w_2^k}{\sqrt{\lambda(j-\mu,k,0)}}dw_1}_{(j,k)\in Z_1}$ is an orthonormal basis for the span of $dw_1$ in $A^2_{1,0}(D_\mu;dV_{D_\mu})$.  Using \eqref{eq:theta_w_1_inverse} to express $dw_1$ as a multiple of the unit-length vector $\theta^1_w$, we obtain
  \[
    (P_{D_\mu}^1(f\theta^1_w))(w)=-\sum_{(j,k)\in Z_1}\int_{D_\mu} f(u)\frac{\bar u_1^{j}\bar u_2^ke^{i\log|u_2|^2}}{|u_1|^\mu\sqrt{\lambda(j-\mu,k,0)}}dV_{D_\mu}\frac{2\mu w_1^{j-1} w_2^k}{\sqrt{\lambda(j-\mu,k,0)}}dw_1.
  \]
  If $j$ is an integer and $j>0$, then $j\geq 1$ so \eqref{eq:s_range_positive_2} implies that $j>\mu s$.  Hence, Lemma \ref{lem:orthonormal_basis} implies that for $s$ satisfying \eqref{eq:s_range_positive_2} $\set{\frac{2\mu w_1^{j-1} w_2^k}{\sqrt{\lambda(j-\mu,k,s)}}dw_1}_{(j,k)\in Z_1}$ is an orthonormal basis for the span of $dw_1$ in $A^2(D_\mu;\delta^{-2s}dV_{D_\mu})$, and so
  \begin{multline*}
    \norm{P_{D_\mu}^1(f\theta^1_w)}^2_{L^2(D_\mu;\delta^{-2s}dV_{D_\mu})}=\\
    \sum_{(j,k)\in Z_1}\abs{\int_{D_\mu} f(w)\frac{\bar w_1^{j}\bar w_2^ke^{i\log|w_2|^2}}{|w_1|^\mu\sqrt{\lambda(j-\mu,k,0)}}dV_{D_\mu}}^2\frac{\lambda(j-\mu,k,s)}{\lambda(j-\mu,k,0)}.
  \end{multline*}
  As before, \eqref{eq:lambda_normalized_bounds} implies that there exists a constant $\tilde C_{\mu,s}>1$ such that
  \[
    \frac{\lambda(j-\mu,k,s)\lambda(j-\mu,k,-s)}{(\lambda(j-\mu,k,0))^2}\leq\tilde C_{\mu,s}
  \]
  for all $(j,k)\in Z_1$, and hence
  \begin{multline*}
    \norm{P_{D_\mu}^1(f\theta^1_w)}^2_{L^2(D_\mu;\delta^{-2s}dV_{D_\mu})}\leq\\
    \tilde C_{\mu,s}
    \sum_{(j,k)\in Z_1}\abs{\int_{D_\mu} f(w)\frac{\bar w_1^{j}\bar w_2^ke^{i\log|w_2|^2}}{|w_1|^\mu\sqrt{\lambda(j-\mu,k,-s)}}dV_{D_\mu}}^2.
  \end{multline*}
  By \eqref{eq:lambda_defn} and \eqref{eq:D_integration_formula_circular_symmetry_extended}, $\set{\frac{\delta^{s}(w)w_1^j w_2^ke^{-i\log|w_2|^2}}{|w_1|^\mu\sqrt{\lambda(j-\mu,k,-s)}}}_{(j,k)\in Z_1}$ is an orthonormal set, so we have
  \[
    \norm{P_{D_\mu}^1(f\theta^1_w)}^2_{L^2(D_\mu;\delta^{-2s}dV_{D_\mu})}\leq\tilde C_{\mu,s}\int_{D_\mu}|f(w)|^2(\delta(w))^{-2s}dV_{D_\mu}.
  \]
  Using \eqref{eq:Sobolev_imbeddings} for the upper bound and \eqref{eq:Sobolev_embeddings_holomorphic} for the lower bound, we have shown that $P_{D_\mu}^1$ is continuous from the span of $\theta^1_w$ in $W^s_{1,0}(D_\mu;dV_{D_\mu})$ to $W^s_{1,0}(D_\mu;dV_{D_\mu})$ whenever \eqref{eq:s_range_positive_2} holds.
  
  Suppose that \eqref{eq:s_range_positive_2} holds.  Since $\lfloor\mu\rfloor\leq\mu$, we have $\frac{1-\lfloor\mu\rfloor}{\mu}+1\geq\frac{1}{\mu}$, and hence \eqref{eq:s_range_positive_2} implies \eqref{eq:s_range_positive}.  For $f\in W^s_{1,0}(D_\mu;dV_{D_\mu})$, we have the decomposition $f=f_1\theta^1_w+f_2\theta^2_w$, where $f_1,f_2\in W^s(D_\mu;dV_{D_\mu})$.  If we compare \eqref{eq:orthonormal_basis} with \eqref{eq:orthonormal_basis_one_form}, we see that
  \begin{equation}
  \label{eq:Bergman_projection_decomposition_one_form}
    P_{D_\mu}^1f=P_{D_\mu}^1(f_1\theta^1_w)+(P_{D_\mu}^0f_2)\theta^2_w.
  \end{equation}
  Since we have both \eqref{eq:s_range_positive} and \eqref{eq:s_range_positive_2}, we have already shown that both terms in \eqref{eq:Bergman_projection_decomposition_one_form} are bounded in the $W^s_{1,0}(D_\mu;dV_{D_\mu})$ norm by $\norm{f}_{W^s_{1,0}(D_\mu;dV_{D_\mu})}$, and so the $p=1$ case is proven.
  
  Suppose once more that \eqref{eq:s_range_positive_2} holds.  For $f\in W^s_{2,0}(D_\mu;dV_{D_\mu})$, we have the decomposition $f=f_0\theta^1_w\wedge\theta^2_w$, where $f_0\in W^s(D_\mu;dV_{D_\mu})$.  Comparing \eqref{eq:orthonormal_basis_one_form} and \eqref{eq:orthonormal_basis_two_form}, we have
  \begin{equation}
  \label{eq:Bergman_projection_decomposition_two_form}
    P_{D_\mu}^2f=P_{D_\mu}^1(f_0\theta^1_w)\wedge\theta^2_w.
  \end{equation}
  Once again, we have already proven the necessary estimate to show that \eqref{eq:Bergman_projection_decomposition_two_form} implies continuity for $P_{D_\mu}^2$ in $W^s_{2,0}(D_\mu;dV_{D_\mu})$.

\end{proof}

\begin{thm}
\label{thm:irregularity}
  For
  \begin{equation}
  \label{eq:s_range_negative}
    \frac{1-\lfloor\mu\rfloor}{\mu}+1\leq s<\frac{1}{2},
  \end{equation}
  there exists $f\in C^\infty(\overline{\Omega_\mu})$ such that $P_{\Omega_\mu}^0 f\notin W^s(\Omega_\mu)$.  For
  \begin{equation}
  \label{eq:s_range_negative_2}
    \frac{1}{\mu}\leq s<\frac{1}{2},
  \end{equation}
   and $p\in\{1,2\}$, there exists $f\in C^\infty_{p,0}(\overline{\Omega_\mu})$ such that $P_{\Omega_\mu}^p f\notin W^s(\Omega_\mu)$.
\end{thm}

\begin{proof}
  Suppose that \eqref{eq:s_range_negative} holds.  Let $j=1-\lfloor\mu\rfloor$.  Then $j>-\mu$, so $w_1^j\in L^2(D_\mu;dV_{D_\mu})$ by Lemma \ref{lem:orthonormal_basis}.  However, for $s$ satisfying \eqref{eq:s_range_negative}, we have $j\leq\mu(s-1)$, so $w_1^j\notin L^2(D_\mu;\delta^{-2s}dV_{D_\mu})$ by Lemma \ref{lem:orthonormal_basis}, and hence $w_1^j\notin W^s(D_\mu;dV_{D_\mu})$ by \eqref{eq:Sobolev_embeddings_holomorphic}.  Set $\tilde f(w)=\exp\left(-\frac{1}{|w_1|^{\mu}}\right)w_1^j$.  Since $|\tilde f|\leq|w_1|^j$, we still have $\tilde f\in L^2(D_\mu;dV_{D_\mu})$, and hence we may use \eqref{eq:D_integration_formula_circular_symmetry_extended} and Lemma \ref{lem:orthonormal_basis} to compute the Bergman projection
  \[
    P_{D_\mu}^0\tilde f(w)=\frac{1}{\lambda(j,0,0)}\int_{D_\mu}\exp\left(-\frac{1}{|u_1|^{\mu}}\right)|u_1|^{2j}dV_{D_\mu}w_1^j.
  \]
  Since $P_{D_\mu}^0\tilde f$ is a constant multiple of $w_1^j$, $P_{D_\mu}^0\tilde f\notin W^s(D_\mu;dV_{D_\mu})$.

  Define $f\in L^2(\Omega_\mu)$ by $f([z])=\tilde f(\varphi_\mu([z]))$, where $\varphi$ is defined by \eqref{eq:varphi_defn}.  By \eqref{eq:abs_z_1}, $|[z]_1|$ is well-defined independently of the choice of $z$ representing $[z]$, and we have
  \[
    f([z])=\exp\left(-\frac{2}{|[z]_1|}\right)(\varphi_\mu([z]))_1^j.
  \]
  Since \eqref{eq:D_defn} guarantees that $\abs{(\varphi_\mu([z]))_1^j}<1$ on $\Omega_\mu$, $f$ easily extends to $f\in C^\infty(\overline{\Omega_\mu})$ by setting $f([z])=0$ whenever $|[z]_1|=0$.  By Lemma \ref{lem:Sobolev_regularity_equivalence}, we still have $P_{\Omega_\mu}^0f\notin W^s(\Omega_\mu)$.

  Suppose that \eqref{eq:s_range_negative_2} holds.  Then $dw_1\in L^2_{1,0}(D_\mu;dV_{D_\mu})$ by Lemma \ref{lem:orthonormal_basis}.  However, for $s$ satisfying \eqref{eq:s_range_negative}, we have $1\leq\mu s$, so $dw_1\notin L^2_{1,0}(D_\mu;\delta^{-2s}dV_{D_\mu})$ by Lemma \ref{lem:orthonormal_basis}, and hence $dw_1\notin W^s_{1,0}(D_\mu;dV_{D_\mu})$ by \eqref{eq:Sobolev_embeddings_holomorphic}.  Set $\tilde f_1(w)=\exp\left(-\frac{1}{|w_1|^{\mu}}\right)dw_1$.  Since $|\tilde f_1|\leq|dw_1|$, we still have $\tilde f_1\in L^2_{1,0}(D_\mu;dV_{D_\mu})$, and hence we may use \eqref{eq:theta_w_1_inverse}, \eqref{eq:D_integration_formula_circular_symmetry_extended}, and Lemma \ref{lem:orthonormal_basis} to compute the Bergman projection
  \[
    P_{D_\mu}^1\tilde f_1(w)=\frac{1}{\lambda(1-\mu,0,0)}\int_{D_\mu}\exp\left(-\frac{1}{|u_1|^{\mu}}\right)|u_1|^{2-2\mu}dV_{D_\mu}dw_1.
  \]
  Since $P_{D_\mu}^1\tilde f_1$ is a constant multiple of $dw_1$, $P_{D_\mu}^1\tilde f_1\notin W^s_{1,0}(D_\mu;dV_{D_\mu})$.  As before, we may pull this back to obtain a counterexample in $\Omega_\mu$.  If we set $\tilde f_2=\tilde f_1\wedge\theta_w^2$, then we may use \eqref{eq:Bergman_projection_decomposition_two_form} to see that we also have a counterexample when $p=2$.
\end{proof}

\begin{cor}
\label{cor:sharp_regularity}
  For $0<r<\frac{1}{2}$ and $p\in\{0,1,2\}$, there exists $\mu>2$ such that $P_{\Omega_\mu}^p$ is continuous in $W^s(\Omega_\mu)$ for all $0\leq s<r$ and $P_{\Omega_\mu}^p$ is not continuous in $W^r(\Omega_\mu)$.
\end{cor}

\begin{proof}
  Suppose $p=0$.  Given $0<r<\frac{1}{2}$, set $\ell=\left\lceil\frac{1}{r}\right\rceil$ and $\mu=\frac{\ell-1}{1-r}$.  We have $\mu-\ell=\frac{r\ell-1}{1-r}$.  Since $\frac{1}{r}\leq\ell<\frac{1}{r}+1$, we have $0\leq\mu-\ell<\frac{r}{1-r}<1$, and hence $\lfloor\mu\rfloor=\ell$.  This means that
  \[
    \frac{1-\lfloor\mu\rfloor}{\mu}+1=\frac{(1-\ell)(1-r)}{\ell-1}+1=r.
  \]
  By Theorems \ref{thm:regularity} and \ref{thm:irregularity}, we are done.
  
  When $p\in\{1,2\}$, we may set $\mu=\frac{1}{r}$ and Theorems \ref{thm:regularity} and \ref{thm:irregularity} will give us the same result.
\end{proof}

\bibliographystyle{amsplain}
\bibliography{harrington}
\end{document}